\documentclass[a4paper,12pt]{amsart}
\usepackage{amssymb,mathrsfs,comment}
\usepackage{eepic}
\usepackage[dvipdfm]{graphicx}
\usepackage[dvipdfm,setpagesize=false]{hyperref}
\makeatletter
\theoremstyle{plain}
  \newtheorem{thm}{Theorem}[section]
  \newtheorem{prop}[thm]{Proposition}
  \newtheorem{lem}[thm]{Lemma}
  \newtheorem{cor}[thm]{Corollary}
  
\theoremstyle{definition}
  \newtheorem{dfn}[thm]{Definition}
  \newtheorem{exmp}[thm]{Example}
  \newtheorem{ques}[thm]{Question}
  \newtheorem{rem}[thm]{Remark} 

\numberwithin{equation}{section}

\let\opn\operatorname 
\newcommand\ba{\mathbf{a}}
\newcommand\BB{\mathbb{B}}
\newcommand\NN{\mathbb{N}}
\newcommand\ZZ{\mathbb{Z}}
\newcommand\kk{\Bbbk}
\newcommand\wS{\widetilde{S}}
\newcommand\wI{\widetilde{I}}

\newcommand\wF{\widetilde{F}}
\newcommand\wG{\widetilde{G}}

\newcommand\wP{{\widetilde{P}}}

\newcommand\m{\mathsf{m}}
\newcommand\n{\mathsf{n}}

\newcommand\wm{\widetilde{\m}}
\newcommand\wn{\widetilde{\n}}
\newcommand\wu{\widetilde{\mathsf{u}}}
\newcommand\fb{\mathfrak{b}}
\newcommand\ga{\alpha}
\newcommand\gl{\lambda}

\newcommand\gs{\sigma}
\newcommand\gt{\tau}
\newcommand\gu{\upsilon}
\newcommand\gD{\Delta}
\newcommand\gG{\Gamma}
\newcommand\gL{\Lambda}
\newcommand\gU{\Upsilon}
\newcommand\wgD{\widetilde{\gD}}
\newcommand\cC{\mathcal{C}}
\newcommand\cF{\mathcal{F}}
\newcommand\sF{\mathscr{F}}
\newcommand\cell{\sigma}

\newcommand\Iff{\Longleftrightarrow}
\newcommand\void\varnothing
\newcommand\pr[1]{\left(#1\right)}
\newcommand\bra[1]{\left\{ #1\right\}}

\newcommand\set[2]{\bra{ #1\ \left|\ #2\right. }}
\newcommand\ang[1]{\langle#1\rangle}
\newcommand\lex{{\opn{lex}}}

\newcommand\supp{\opn{supp}}
\newcommand\bpol{\opn{\mathsf{b-pol}}}
\newcommand\codim{\opn{codim}}
\newcommand\chara{\opn{char}}
\newcommand\lcm{\opn{lcm}}

\newcommand\Ass{{\opn{Ass}}}
\newcommand\gr{\opn{gr}}
\newcommand\intr{\opn{int}}
\newcommand\ijseq[2][1]{(i_{#1},j_{#1}),\dots ,(i_{#2},j_{#2})}
\title[CW complexes supporting Eliahou-Kervaire type resolutions]{On  CW complexes supporting Eliahou-Kervaire type resolutions of Borel fixed ideals}
\author{Ryota Okazaki}
\thanks{The first author is supported by JSPS KAKENHI Grant Number 20624109}
\address{Faculty of Education, Fukuoka University of Education,
Munakata, Fukuoka 811-4192, Japan}
\email{rokazaki@fukuoka-edu.ac.jp}
\author{Kohji Yanagawa}
\thanks{The second author is supported by JSPS KAKENHI Grant Number 22540057}
\address{Department of Mathematics, Kansai University,
Suita 564-8680, Japan}
\email{yanagawa@kansai-u.ac.jp}
\begin{document}

  \begin{abstract}
We prove that the Eliahou-Kervaire resolution of
a Cohen-Macaulay stable monomial is supported by a regular CW complex 
whose underlying space is a closed ball. 
We also show that the modified Eliahou-Kervaire resolution of
variants of a  Borel fixed  ideal (e.g., a squarefree strongly stable ideal)   
are supported by regular CW complexes, and their underlying spaces are closed balls in the  Cohen-Macaulay case. 
\end{abstract}

\maketitle

\section{introduction}\label{sec:intro}
Throughout this paper, let $\kk$ be a field, and $S$ the polynomial ring
$\kk[x_1,\dots ,x_n]$. 
Free resolutions of monomial ideals of $S$ (or free resolutions of more complicated objects) 
 sometimes admit structure given by CW complexes.
Such resolutions are called {\em cellular resolutions}. 
Since the initiative works by Bayer, Peeva, and Strumfels \cite{BPS, BS},
they have been intensely studied, see for example,  \cite{BaWe,C,M,NR,S, V}.

Let us recall precise definitions. 
For a CW complex $X$, let $X^{(i)}$ denote
the set of the $i$-cells of $X$, and set $X^{(*)} := \bigcup_{i \ge 0} X^{(i)}$.
The set $X^{(*)}$ of all the cells is regarded as the poset with the order 
defined by $c < c'$ if $c$ is contained in the closure of $c'$.
Given an order-preserving map $\gr:X^{(*)} \to \ZZ^n_{\ge 0}$, where
the order of $\ZZ^n_{\ge 0}$ is given by componentwise-comparing,
we construct a $\ZZ^n$-graded chain complex $\cF^{(X,\gr)}_\bullet$ of $S$-free modules as follows:
set $\cF^{(X,\gr)}_i := \bigoplus_{c \in X^{(i)}} S \cdot e(c)$, where $e(c)$ is an $S$-free
base of degree $\gr(c)$, and define the differential map $\partial^{(X,\gr)}_i : \cF^{(X,\gr)}_i \to \cF^{(X,\gr)}_{i-1}$ by
$$
e(c) \longmapsto  \sum_{c' \in X^{(i-1)}} [c:c'] \cdot x^{\gr(c) - \gr(c')} \cdot e(c'), 
$$
where we set $x^{\ba} = \prod_{i=1}^nx_i^{a_i} \in S$ for $\ba = (a_1,\dots ,a_n) \in \ZZ^n_{\ge 0}$ and $[c:c'] \in \ZZ$ denotes the coefficient of $c'$ in the image
of $c$ by the differential map in the cellular homology of $X$.
A $\ZZ^n$-graded $S$-free resolution $\cF_\bullet$ of  some module is said to be {\em cellular} and {\em supported by} $X$
if there exists a CW complex $X$ and a map $\gr: X^{(*)} \to \ZZ^n_{\ge 0}$ such that 
$\cF_\bullet = \cF^{(x,\gr)}_\bullet$. 
The cellularity of $\cF_\bullet$ is a problem on (the existence of) choices of free bases, but this notion arises mysterious phenomena as stated below.   

While there is a monomial ideal whose minimal free resolution cannot be cellular (\cite{V}), 
Batzies and Welker \cite{BaWe} showed that minimal free resolutions of  most of  ``famous" 
monomial ideals admit  cellular structure given through Forman's {\it discrete Morse theory} (\cite{F}). 
However, their approach does not tell us much about the supporting CW complex $X$.  
For example, it is very hard to check that  $X$ is {\it regular}. 
See Definition~\ref{dfn:CW} for the definition of regular CW complexes, but we just remark that the regularity is  a natural requirement from 
combinatorial view point. 

A monomial ideal $I \subset S$ is called {\it Borel fixed}, if $x_i \cdot (\m/x_j) \in I$ for
any monomial $\m \in I$ and $i, j \in \ZZ$ with $i < j$ and $x_j \mid \m$.
Here we use this terminology even if $\chara(\kk) >0$.
Borel fixed ideals are very important, since 
they appear as the generic initial ideals of homogeneous ideals (if $\chara(\kk)=0$). 
A monomial ideal $I$ is called {\em stable}, if
$x_i \cdot (\m/x_k) \in I$ for any $\m \in I$ and $i < k:= \max\set{j}{x_j \text{ divides } \m}$. 
Clearly, Borel fixed ideals are stable.
In their  influential paper \cite{EK}, Eliahou and Kervaire explicitly constructed  minimal free resolutions 
of stable monomial ideals. 
Recently, Mermin \cite{M} and Clark \cite{C} showed that 
the Eliahou-Kervaire resolution (more precisely, their choice of free bases) 
is supported by a regular CW complex.

The first main result of the present paper is the following.  

\medskip

\noindent{\bf Theorem~\ref{thm:EKball}.}
{\it Let $I$ be a Cohen-Macaulay stable monomial ideal. 
Then the Eliahou-Kervaire resolution of $I$ is supported by a regular CW complex whose underlying space is 
a closed ball.}

\medskip

A {\it squarefree strongly stable (monomial) ideal} is a squarefree analog of a Borel fixed ideal, and also important in combinatorial 
commutative algebra (c.f. \cite{AHH2}). For a Borel fixed ideal $I \subset S$, we have the corresponding  
squarefree strongly stable ideal $I^\gs$ of a larger polynomial ring 
$T=\kk[x_1, \ldots, x_N]$ with $N \gg 0$, and any squarefree strongly stable ideal is given in this way. 

Let $I$ be a Borel fixed ideal,
and $G(I)$ the minimal set of monomial generators of $I$. 
Take $d \in \NN$ so that $\deg (\m) \le d$ for all $\m \in G(I)$, and 
set $$\wS :=\kk[x_{i,j} \mid 1\le i \le n,\ 1\le j\le d].$$
Consider the subsets  
$\Theta := \{ x_{i,1}-x_{i,j} \mid 1 \leq i \leq n, \, 2 \leq j \leq d \, \}$ and  
$\Theta' := \{ \, x_{i,j} -x_{i+1, j-1} \mid 1 \le i <n,   1< j \le d \, \}$ of $\wS$. 
For $\m \in G(I)$ such that
$\m = \prod_{i=1}^{\deg(\m)}x_{\alpha_i}$ with
$\alpha_1 \le \alpha_2 \le \dots \le \alpha_{\deg(\m)}$, we set
$$
\bpol(\m) := \prod_{i=1}^{\deg(\m)}x_{\alpha_i,i} \in \wS. 
$$
Now we have  the non-standard polarization 
$$
\bpol(I) := (\ \bpol(\m) \mid \m \in G(I)\ ) \subset \wS
$$
 of $I$. 
For example, if $I$ is a Borel fixed ideal  $(x_1^2, x_1x_2, x_2^3)$, 
then we have  $\bpol(I) = (x_{1,1}x_{1,2}, x_{1,1}x_{2,2},  x_{2,1}x_{2,2}x_{2,3})$ and 
$I^\gs = (x_1x_2, x_1x_3, x_2x_3x_4)$.

In \cite{Y}, generalizing results of Nagel and Reiner \cite{NR}, the second 
author showed that both $\Theta$ and $\Theta'$ form  regular sequences of $\wS/\bpol(I)$,  and $\wS/(\theta) \otimes_{\wS} I \cong I$ 
(resp.  $\wS/(\theta') \otimes_{\wS} I \cong I^\gs$) as $S$-modules 
(resp. $T$-modules)  
via the ring isomorphism $\wS/(\theta) \cong S$ (resp.  $\wS/(\theta') \cong T$).  In particular, $\bpol(I)$ is actually a polarization. 
However, the Eliahou-Kervaire resolution of $I$ cannot be lifted to $\bpol(I)$ directly. 
 
In the previous paper \cite{OY}, generalizing  \cite{NR}, we explicitly constructed the minimal free
resolution $\wP_\bullet$ of $\bpol(I)$, 
and showed that $\wP_\bullet$ is supported by a CW complex given by the discrete Morse theory. 
Note that  $\wS/(\theta) \otimes_{\wS} \wP_\bullet$ and $\wS/(\theta') \otimes_{\wS} \wP_\bullet$ are  
minimal free resolutions of $I$ itself and  $I^{\gs}$ respectively, and both are supported by the same CW complex as $\wP_\bullet$. 
We call these resolutions of $I$, $I^\gs$ and $\bpol(I)$ the {\it modified Eliahou-Kervaire resolutions}. 
While we could not show the regularity of the support of $\wP_\bullet$ in \cite{OY}, we will prove the following. 

\medskip

\noindent{\bf Corollary~\ref{cor:EKtype_reg} and Theorem~\ref{thm:ball}.}
{\it Let $I \subset S$ be a Borel fixed ideal, $\bpol(I)$ its polarization defined above, and 
$\wP_\bullet$ the modified Eliahou-Kervaire resolution of $\bpol(I)$. 
Then $\wP_\bullet$ is supported by a regular CW complex $X$. 
Moreover, if $S/I$ is Cohen-Macaulay, then the underlying space of $X$ is a closed ball. 

Clearly, the modified Eliahou-Kervaire resolutions of $I$ and $I^\gs$ have the same property (see \S6). }

\medskip

For the proofs of all main results of the present paper, Clark's idea using the EL shellability plays a key role. 
\section{Preliminaries}\label{sec:prel}

Throughout this section, $I$ denotes a stable monomial ideal of $S$
(see Section \ref{sec:intro} for the definition).
We shall recall the construction of the Eliahou-Kervaire resolution of $I$.
Following usual notation, for a monomial $\m$ of $S$,
we set $\supp(\m) := \set{i}{x_i \text{ divides } \m}$, $\max(\m) := \max(\supp(\m))$,
and $\min(\m) := \min(\supp(\m))$.

\begin{lem}[{\cite[Lemma1.1]{EK}}]\label{lem:g}
For any monomial $\m \in I$, there exists a unique $\m_0 \in G(I)$ dividing $\m$
such that
$$
\max(\m_0) \le \min(\m/\m_0).
$$
\end{lem}

Following \cite{EK}, let $g^I$, or simply $g$, denotes the function
that sends any monomial $\m \in I$ to $\m_0 \in G(I)$.
A pair $(F,\m)$ of a subset $F \subset \NN$ and a monomial $\m \in G(I)$ is said to be {\em admissible for $I$} if $F = \void$, or otherwise
$F = \bra{i_1,\dots,i_q}$ with
$$
1 \le i_1 < \dots < i_q < \max(\m).
$$

Following usual convention, for $F \subseteq \bra{1,\dots ,n}$,
let $x_F$ denote the monomial $\prod_{i \in F}x_i$.
For $q \ge 0$, set
$$
A^I_q := \set{(F,\m)}{\#F = q,\ \text{$(F,\m)$ is admissible}}.
$$
For $F = \bra{i_1,\dots ,i_q}$, let $F_{\ang{i_r}}$ denote the set $F \setminus \bra{i_r}$,
and for $\m \in G(I)$ such that $(F, \m) \in A^I_q$, set
$$
B^I(F,\m) := \set{i_r}{(F_{\ang{i_r}},g(x_{i_r}\m)) \in A^I_{q-1}}.
$$
The $q$-th component of the Eliahou-Kervaire resolution $P_\bullet$ of $I$ is
$$
P_q := \bigoplus_{(F,\m) \in A^I_q} S \cdot e(F,\m),
$$
where $e(F,\m)$ is the $S$-free basis with the same multi-degree (with respect to $\ZZ^n$-grading) as $x_F \cdot \m$.
The differential maps $\partial: P_q \to P_{q-1}$ are defined as follows: for $(F,\m) \in A^I_q$ with $F = (i_1,\dots ,i_q)$ and $q \ge 1$,
$$
\partial(e(F,\m)) := \sum_{r=1}^q (-1)^r x_{i_r} \cdot e(F_{\ang{i_r}},\m)
  - \sum_{i_r \in B^I(F,\m)} (-1)^r \frac{x_{i_r}\m}{g(x_{i_r}\m)} \cdot e(F_{\ang{i_r}},g(x_{i_r}\m)).
$$

\begin{thm}[{\cite[Theorem~2.1]{EK}}]
The above $P_\bullet$ is a minimal $\ZZ^n$-graded free resolution of $I$.
\end{thm}

In his paper \cite{M}, Mermin showed that the Eliahou-Kervaire resolutions are cellular
and supported by regular CW complexes.
Clark \cite{C} also proved this result by detecting {\em EL-shellability}
of each interval of the poset associated to the resolution.
In Section \ref{sec:regular}, we will make use of his technique to show the
modified Eliahou-Kervaire resolution of $\bpol(I)$ for a Borel fixed ideal $I$ is supported by
a regular CW complex. That's why we will introduce the argument by Clark
in the below.

First, let us recall the basic notion and properties of a simplicial complex.
By definition, an (abstract) simplicial complex
$\gD$ on the finite vertex set $V$ is a subset of the power set $2^V$
which is closed under taking subsets (i.e., for $\gs$, $\gt \subset V$,
$\gs \in \gD$ and $\gt \subseteq \gs$ imply $\gt \in \gD$).
The elements of the simplicial complex $\gD$ are called the {\em faces} of $\gD$,
and the {\em dimension} of a face $\gs \in \gD$ is defined to be
$\# \gs - 1$. The dimension of a face $\gs$ is denoted by $\dim \gs$.
A face of dimension $d$ is called a $d$-face.
The {\em dimension} $\dim \gD$ of $\gD$ is, by definition,
the maximum of the dimensions of the faces of $\gD$.
The faces of $\gD$ which are maximal with respect to inclusion
are called the {\em facets} of $\gD$. The set of the facets of $\gD$
is denoted by $\sF(\gD)$. Clearly $\sF(\gD)$ characterizes
$\gD$ completely. If all the facets of $\gD$ have the same dimension,
then $\gD$ is said to be {\em pure}.
For a subset $\bra{\gs_1,\dots ,\gs_r}$ of $2^V$, let $\ang{\gs_1,\dots ,\gs_r}$
denote the smallest simplicial complex containing $\bra{\gs_1,\dots ,\gs_r}$.
When $\gs = \bra{v}$ for some $v \in V$, the simplicial complex $\ang{\gs}$
(the $0$-simplex on $\bra{v}$)
is simply written as $\ang{v}$.

A pure simplicial complex $\gD$ with $\dim \gD = d$ is said to be {\em shellable} if
there exists a linear ordering $\gs_1,\dots ,\gs_r$ on $\sF(\gD)$ such that
the intersection $\ang{\gs_1,\dots ,\gs_{i-1}} \cap \ang{\gs_i}$ is pure of dimension $d-1$
for all $i$ with $2 \le i \le r$.
Though shellability can be defined for non-pure simplicial complexes as is known well,
in this paper, we do not need it.
Thus we refer to pure shellable simplicial complexes just as shellable ones.
See \cite{St} for the general definition and the other equivalent conditions of shellability.

It is well known that a shellable simplicial complex $\gD$ is {\em Cohen-Macaulay}
over any field $\kk$, i.e., its Stanley-Reisner ring $\kk[\gD]$ is Cohen-Macaulay
for any field $\kk$ (see \cite{BH, St} for details).
Constructible simplicial complexes are those in the hierarchy between
shellability and Cohen-Macaulay-ness. A {\em constructible} simplicial complex
$\gD$ is the one obtained by the following recursive procedure:
\begin{enumerate}
\item any simplex, i.e., a simplicial complex with the only one facet,
is constructible:
\item for any two constructible simplicial complexes $\gD$ and $\gD'$ of dimension $d$,
if the intersection $\gD \cap \gD'$ is constructible of dimension $d-1$, then
the union $\gD \cup \gD'$ is also constructible (of dimension $d$).
\end{enumerate}
It is easy to verify that a shellable simplicial complex is constructible.

We will make use of the following proposition to show a given simplicial complex
is homeomorphic to a closed ball.

\begin{prop}[cf. {\cite[Theorem 11.4]{B95}, \cite[Proposition 1.2]{DK}}]\label{prop:balllem}
A geometric realization of a $d$-dimensional simplicial complex $\gD$ is homeomorphic to a closed ball if $\gD$ satisfies the following conditions:
\begin{enumerate}
\item $\gD$ is constructible: \label{enum:const}
\item every $(d-1)$-face is contained in at most two $d$-faces: \label{enum:atmost2}
\item there exists a $(d-1)$-face contained in only one $d$-face. \label{enum:only1}
\end{enumerate}
\end{prop}

For a simplicial complex $\gD$ and a new vertex $v$,
the {\em join} $\gD \ast \ang{v}$ is, by definition, the subset
$\gD \cup \set{\gs \cup \bra{v}}{\gs \in \gD}$ of $2^{V \cup \bra{v}}$.
Clearly, it becomes a simplicial complex on $V \cup \bra{v}$.
Though the following may be known well, we give its proof for completeness.

\begin{lem}\label{lem:const_lem}
Let $\gD$ be a simplicial complex, $v$ a new vertex.
If the join $\gD \ast \ang{v}$ is constructible, then so is $\gD$ itself.
\end{lem}
\begin{proof}
Set $d := \dim \gD$.
Note that there exists the following one-to-one corresponding between $\sF(\gD)$
and $\sF(\gD \ast \ang{v})$
\begin{equation}
\sF(\gD) \ni \gs \ \longleftrightarrow\  \gs \cup \bra{v} \in \sF(\gD \ast \ang{v}).
\tag{$\ast$}
\end{equation}
Furthermore the following holds:
$\gU = (\gD \cap \gU) \ast \ang{v}$ for any subcomplex $\gU$ of $\gD \ast \ang{v}$
such that $v \in \gs$ for all $\gs \in \sF(\gU)$.
Indeed, by the hypothesis for $\gU$, 
$\sF(\gD \cap \gU) = \set{\gs \setminus \bra{v}}{\gs \in \sF(\gU)}$,
hence it follows that $\sF((\gD \cap \gU) \ast \ang{v}) = \sF(\gU)$,
which implies $\gU = (\gD \cap \gU) \ast \ang{v}$.

Now assuming $\gD \ast \ang{v}$ is constructible, we shall show $\gD$
is also constructible by the induction on $\#\sF(\gD)$ and $d$.
In the case where $\#\sF(\gD) = 1$ or $d  = 0$, the complex $\gD$ is always
constructible. Assume $\#\sF(\gD) > 1$ and $d > 0$.
By the hypothesis, there exist constructible subcomplexes $\wgD_1$, $\wgD_2$ of
$\gD \ast \ang{v}$ with $\dim \wgD_1 = \dim \wgD_2 = \dim(\gD \ast \ang{v}) = d + 1$
such that $\wgD_1 \cap \wgD_2$ is constructible of dimension $d$.
We set $\gD_i := \gD \cap \wgD_i$. By ($\ast$), every facet of $\wgD_1$ and $\wgD_2$
contains $v$ and hence so does $\wgD_1 \cap \wgD_2$.
Thus $\gD_1 \ast \ang{v} = \wgD_1$, $\gD_2 \ast \ang{v} = \wgD_2$,
and $(\gD_1 \cap \gD_2) \ast \ang{v} = \wgD_1 \cap \wgD_2$
(Note that $\gD_1 \cap \gD_2 = \gD \cap (\wgD_1 \cap \wgD_2)$).
Applying the inductive hypothesis implies $\gD_1$, $\gD_2$ and
$\gD_1 \cap \gD_2$ are constructible.
Clearly $\dim \gD_1 = \dim \gD_2 = d$ and $\dim(\gD_1 \cap \gD_2) = d- 1$.
Therefore $\gD = \gD_1 \cup \gD_2$ is constructible.
\end{proof}

Next, we will recall the basic notion and properties of posets (partially ordered sets).
From here to the end of the paper, unless otherwise specified,
a poset means a {\em finite} poset, i.e., a poset with finite cardinality.
A poset $\gG$ is said to be {\em pure} if every maximal chain in $\gG$
has the same length,
and {\em bounded} if $\gG$ has the greatest element $\hat 1$ and the least one $\hat 0$.
Let $\gD(\gG)$ denote the {\em order complex} of $\gG$, i.e., the simplicial complex
on $\gG$ consisting of the chains in $\gG$,
where each chain is considered as a subset of $\gG$, ignoring the order.
Clearly, $\gG$ is pure if and only if $\gD(\gG)$ is pure.
For elements $\gs, \gt \in \gG$, we set
$$
[\gs, \gt]_{\gG} := \set{\gu \in \gG}{\gs \le \gu \le \gt},
$$
and a subposet as above is called an {\em interval} in $\gG$.
If there is no fear of confusion, the interval $[\gs,\gt]_{\gG}$ is simply
denoted by $[\gs,\gt]$.
Clearly an interval in a pure poset is also pure.
The {\em length} of an interval $[\gs,\gt]$ is by definition the maximum of the length of
the chains in $[\gs,\gt]$.
A pure poset is said to be {\em thin} if 
every interval of length $2$ has cardinality $4$.
A poset $\gG$ is said to be {\em shellable} if so is its order complex $\gD(\gG)$.

There is a well-known method to judge shellability
by means of a labeling on the chains due to Bj\"orner \cite{B80}.
Let $\gG$ be a poset. For $\gs ,\gt \in \gG$,
we will write $\gt \lessdot \gs$ if $\gt < \gs$ and $[\gs,\gt] = \bra{\gs,\gt}$.
We define $\cC^q(\gG)$ to be the set of the unrefinable chains
$c_0 \lessdot c_1 \lessdot \cdots \lessdot c_q$ of length $q$ in $\gG$.
A map $\gl: \cC^1(\gG) \to \gL$, where $\gL$ is some poset, is called
an {\em edge labeling} of $\gG$.
For any positve integer $q$, an edge labeling $\gl$ is extended to the map
from $\cC^q(\gG)$ to $\gL^q$ sending each
$c_0 \lessdot c_1 \lessdot \cdots \lessdot c_q$ to
$$
(\gl(c_0 \lessdot c_1) , \cdots ,\gl(c_{q-1} \lessdot c_q)).
$$
The extended map is also denoted by $\gl$ by abuse of notation.
Though $\gl(c)$ is thus an ordered tuple,
we will use the notation $i \in \gl(c)$ to mean that $i$ appears in $\gl(c)$
for convenience.

For $c,c' \in \cC^q(\gG)$, we write $c <_\lex c'$
whenever $\gl(c) < \gl(c')$ lexicographically with respect to the order on $\gL$.
An unrefinable chain $c_0 \lessdot c_1 \lessdot \cdots \lessdot c_q$ is
said to be {\em increasing} if
$$
\gl(c_0 \lessdot c_1) \le \gl(c_1 \lessdot c_2) \le \cdots \le \gl(c_{q-1} \lessdot c_q).
$$

\begin{dfn}[{\cite[Definition 2.1]{BjWa}}]\label{dfn:EL-labeling}
For a poset $\gL$ and a bounded pure one $\gG$,
an edge labeling $\gl: \cC^1(\gG) \to \gL$
is called an {\em EL-labeling} if for every interval $[\gs ,\gt]$ in $\gG$,
\begin{enumerate}
\item there is a unique increasing maximal chain $c$ in $[\gs,\gt]$, and
\item $c <_\lex c'$ for any other maximal chain $c'$ in $[\gs,\gt]$.
\end{enumerate}
A poset possessing an EL-labeling is said to be {\em EL-shellable}.
\end{dfn}

\begin{prop}[{\cite[Theorem 2.3]{B80}, \cite[Proposition 2.3]{BjWa}}]
A bounded pure poset is shellable if it is EL-shellable.
\end{prop}

Now let us recall the definition of a CW complex and its regularity.
Since we treat only a finite CW complex in this paper,
we restrict ourselves to the finite case.
For the definition of a general CW complex, see \cite{LW}.
Let $\BB^d$ denote the $d$-dimensional closed ball,
$\intr(\BB^d)$ its interior, and set $\partial \BB^d := \BB^d \setminus \intr(\BB^d)$.
A subset $\cell$ of a topological space $X$ is called
an (open) {\em $d$-cell} if there exists a $d$-dimensional closed ball $\BB^d$
and a continuous map $f_\cell: \BB^d \to X$ whose restriction to $\intr(\BB^d)$
induces a homeomorphism onto $\cell$.
For such cell $\cell$, we set $\dim \cell := d$ and refer to the continuous map $f_\cell$
as the {\em characteristic map} of $\cell$.

\begin{dfn}\label{dfn:CW}
A {\em finite CW complex} is a Hausdorff space $X$
together with a finite set $X^{(\ast)}$ of cells of $X$ satisfying the following conditions:
\begin{enumerate}
\item $X = \bigcup_{\cell \in X^{(*)}} \cell$:
\item $\cell \cap \cell' = \void$ for $\cell, \cell' \in X^{(\ast)}$ with $\cell \neq \cell'$:
\item for each $\cell \in X^{(d)}$ and its characteristic map $f_\cell$,
$$
f_\cell(\partial \BB^d) \subseteq X^{\le d- 1},
$$
where $X^{(k)} := \set{\cell'' \in X^{(\ast)}}{\dim \cell'' = k}$ and $X^{\le d- 1} := \bigcup_{k \le d- 1} \pr{\bigcup_{\cell'' \in X^{(k)}} \cell''}$.
\end{enumerate}
Moreover if the finite CW complex satisfies the following condition,
then it is said to be {\em regular}:
\begin{enumerate}
\setcounter{enumi}{3}
\item for each $\cell \in X^{(d)}$,
the corresponding characteristic map $f_\cell$ can be chosen to be
a homeomorphism from $\BB^d$ to the closure of $\cell$.
\end{enumerate}
\end{dfn}

From here to the end of the paper, we refer to a finite CW complex simply as a CW complex.

\begin{rem}
In some literatures the empty set $\void$ is considered as
the unique $(-1)$-cell.
Relying on context, we sometimes use this convention.
\end{rem}

The {\em face poset} of a CW complex $(X, X^{(\ast)})$, 
denoted by $\gG_X$, is the poset which is equal to $X^{(\ast)}$ as sets
and whose order is defined as follows: for $\cell, \cell' \in X^{(\ast)}$,
$\cell < \cell'$ if $\cell$ is contained in the closure of $\cell'$.

\begin{prop}[cf. {\cite[Theorem 1.7, Chapter III]{LW}}]\label{prop:faceposet}
For a finite regular CW complex $(X,X^{(\ast)})$ with the $(-1)$-cell $\void$,
(the geometric realization of) the order complex $\gD(\gG_X \setminus \bra{\void})$
is homeomorphic to the underlying space $X$.
\end{prop}

Let $P_\bullet$ be the Eliahou-Kervaire resolution of $I$ as above.
The key idea of the proof that $P_\bullet$ is supported by a regular CW complex
is to show that the poset associated with $P_\bullet$ coincides with
the face poset of some regular CW complex.
Such a poset, i.e., a poset which is isomorphic to a face poset of
some {\em regular} CW complex with the $(-1)$-cell $\void$, is called {\em CW poset},
which is due to Bj\"orner \cite{B84}.

The following is a useful criterion to verify a poset is CW.

\begin{prop}[{\cite[Proposition 2.2]{B84}}]\label{prop:CW}
A thin poset $\gG$ with $\# \gG \ge 2$ is a CW poset if
\begin{enumerate}
\item $\gG$ has a least element $\hat 0$,
\item for any $\gs \in \gG$, the interval $[\hat 0,\gs]$ is finite and shellable.
\label{enum:CW_fin_shell}
\end{enumerate}
\end{prop}

For a poset $\gG$, let $\gG^\ast$ denote its {\em dual poset},
i.e., the poset such that $\gG^\ast$ is equal to $\gG$ as sets
while the order of $\gG^\ast$ is the reverse of that of $\gG$.
It is noteworthy that $\gG$ is shellable if and only if so is $\gG^\ast$.
This is just an immediate consequence of the fact that $\gD(\gG) = \gD(\gG^\ast)$.

Now we are prepared to state Clark's argument. Consider the order on $\bigcup_q A^I_q$
given as the transitive closure of the following one:
for $(F,\m) \in A^I_q$ and $(F',\m') \in A^I_{q-1}$,
$$
(F,\m) > (F',\m') \ \Iff\ (F',\m') = (F_{\ang{i_r}},\m) \text{ or } (F_{\ang{i_r}}, g(x_{i_r}\m))
\quad \exists i_r \in F
$$
With this order, $\bigcup_q A^I_q$ becomes a poset.
Adding the new least element $\hat 0$ to $\bigcup_q A^I_q$, we obtain
the new poset, which is denoted by $\gG_P$.
This is the very poset associated with $P_\bullet$ referred in the above.

Let $\gG_P^\ast$ be the dual poset of $\gG_P$ and $<^\ast$ its order.
Define a labeling $\gl: \cC^1(\gG_P^\ast) \to \ZZ$ as follows:
$\gl((\void,\m) \lessdot^\ast \hat 0) := 0$ for all $\m \in G(I)$ and
$$
\gl((F,\m) \lessdot^\ast (F',\m')) := \begin{cases}
-i_r &\text{if $F' = F_{\ang{i_r}}$ and $\m' = \m$}\\
i_r &\text{if $F'=F_{\ang{i_r}}$ and $\m' = g(x_{i_r}\m)$}.
\end{cases}
$$
Naturally extended labelings $\cC^q(\gG_P^\ast) \to \ZZ^q$ with $q \ge 2$
are also denoted by $\gl$.
For each unrefinable chain $c \in \bigcup_q \cC^q(\gG_P^\ast)$, we set
$$
\gl_+(c) := \set{ k \in \gl(c)}{ k > 0}.
$$

\begin{lem}[{\cite[Theorem 3.6 and Lemma 3.7]{C}}]\label{lem:ELlabel}
The following hold:
\begin{enumerate}
\item The above labeling $\gl$ is an EL-labeling for each intervals in $\gG_P^\ast$:
\label{enum:EKELlabel}
\item for each interval $[(F,\m),(F',\m')]$ with $F \setminus F' = \bra{i_1,\dots, i_q}$,
there exists a unique minimal set $G$ {\rm (}with respect to inclusion{\rm)} such that
$G \subseteq (F \setminus F')$ and $g(x_G \m) = \m'$,
and $G$ coincides with $\gl_+(c)$ of the unique maximal increasing chain $c$ in the interval.
\label{enum:unique_xG}
\end{enumerate}
\end{lem}

By an easy observation, $\gG_P$ is thin,
and hence the above Lemma implies that $\gG_P$ is a CW poset.
Since the $\kk$-coefficients of the images of each differential maps
consist only of $\pm 1$ in the Eliahou-Kervaire resolution $P_\bullet$,
as a corollary, the following can be deduced.

\begin{cor}[{\cite[Theorem 6.4]{C}}, {\cite[Theorem 5.3]{M}}]\label{cor:EK_regular}
The Eliahou-Kervaire resolution $P_\bullet$ is supported by a regular CW complex.
\end{cor}\section{A regular CW complex supporting the Eliahou-Kervaire resolution}
\label{sec:EKball}

As in the previous section, let $I$ be a stable monomial ideal of $S$.
It is quite natural to ask about the topological properties
of a regular CW complex supporting the Eliahou-Kervaire resolution,
while very little is known about such properties.
The goal of this section is to prove that the complex is homeomorphic to
a closed ball if the corresponding Eliahou-Kervaire resolution is
that of a Cohen-Macaulay stable monomial ideal.
Let $\prec$ be the lexicographic order on the set of monomials of $S$
such that $x_1 \succ x_2 \succ \dots \succ x_n$.
Recall that for a stable monomial ideal $I$,
there is the function $g$ which sends a monomial $\m$ in $I$ to
the unique monomial $\m_0 \in G(I)$ which divides $\m$ and satisfies
$\max(\m_0) \le \min(\m/\m_0)$.
Throughout this section, we tacitly use the following properties of $g$
(see \cite[Lemmas 1.2, 1.3, 1.4]{EK}): for any monomial $\m \in G(I)$,
\begin{enumerate}
\item $g(x_i\m) \succeq \m$ for all $i$,
\label{enum:grteq}
\item $g(x_i\m) = \m$ if and only if $ i \ge \max(\m)$, and
\item $g(\m g(\n)) = g(\m\n)$ for any monomial $\m$, $\n$ with $\m \in S$ and $\n \in I$.
\end{enumerate}

\begin{rem}
Though in \cite[Lemma 1.4]{EK}, the converse of the inequality in \eqref{enum:grteq}
is proved for a different order, it is easy to verify that \eqref{enum:grteq} indeed holds
true for our order.
\end{rem}

The function $g$ can be characterized in terms of $\prec$.

\begin{lem}
For a monomial $\m \in I$, the monomial $g(\m)$ is the greatest element, with respect to $\prec$,
among the elements of $G(I)$ dividing $\m$.
\end{lem}
\begin{proof}
Let $\m'$ be the greatest element of $G(I)$ which divides $\m$.
Suppose $\max(\m') > \min(\m/\m')$, and set $i := \min(\m/\m')$.
Then $x_i\m'$ still divides $\m$, and hence so does $g(x_i\m')$.
Since $i < \max(\m')$, the strict inequality $g(x_i\m') \succ \m'$ holds,
which contradicts the maximality of $\m'$.
Therefore it follows that $\max(\m') \le \min(\m/\m')$ and
hence $\m' = g(\m)$ by Lemma \ref{lem:g}.
\end{proof}

\begin{rem}
Let $G(I) = \bra{\m_1,\dots \m_r}$ with $\m_1 \succ \dots \succ \m_r$.
The above lemma also can be deduced by showing $I$ has linear quotients
with respect to the ordering $\m_1,\dots ,\m_r$;
in fact the above lemma says that
the function $g$ is just the {\em decomposition function} of $I$ with
respect to the ordering $\m_1,\dots, \m_r$ (see \cite{HT} for details).
\end{rem}

\begin{lem}\label{lem:lcm}
For an unrefinable chain $c: c_0 \lessdot^\ast c_1 \lessdot^\ast \dots \lessdot^\ast c_q$ in $\gG^\ast_P$, the following holds:
\begin{enumerate}
\item Set $\gl(c) = (\gl_0,\gl_1,\dots,\gl_q)$. If $\gl_i < 0$ for some $i > 0$, then there exists an unrefinable chain $c'$ in $[c_0,c_q]_{\gG^\ast_P}$ with the label $(\gl_i,\gl_0,\gl_1,\dots,\gl_{i-1},\gl_{i+1},\dots,\gl_q)$.
\label{enum:unique_C}
\item
If $c$ is increasing and $c_q \neq \hat 0$, then
$$
x_{\gl_+(c)} = \frac{\lcm(\m,\m')}\m,
$$
where we set $c_0 = (F,\m)$ and $c_q = (F',\m')$.
\label{enum:lcm}
\end{enumerate}
\end{lem}
\begin{proof}
The assertion (1) is easy to prove. We will show only (2).
For simplicity, we set $C := \gl_+(c)$. Obviously $\m' = g(x_C \m)$, and hence $\lcm(\m,\m')$ divides $x_C \m$.
Suppose $\lcm(\m,\m') \neq x_C \m$. Then there exists $C' \subset C$ with $C' \neq C$ such that $\lcm(\m,\m')$ divides $x_{C'} \m$.
Clearly $g(x_C \m)=\m'$ divides $x_{C'}\m$. This implies $g(x_C\m) \preceq g(x_{C'}\m)$.
On the other hand, $g(x_{C'} \m) \preceq g(x_C \m)$ holds, since $g(x_{C'}\m)$ divides $x_C \m$. Therefore it follows that $g(x_C\m) = g(x_{C'}\m)$, which contradicts the assertion
\eqref{enum:unique_xG} of Lemma~\ref{lem:ELlabel}.
\end{proof}

Henceforth we assume $I$ is Cohen-Macaulay with $\codim S/I = h$.
The codimension $h$ is equal to
the projective dimension of $S/I$ by Auslander-Buchsbaums's formula.
Moreover $\Ass(S/I) = \bra{(x_1,\dots ,x_h)}$, since any graded associated prime
ideal of a stable monomial ideal is of the form $(x_1,\dots ,x_k)$ for some $k$.
Thus the following holds.

\begin{lem}\label{lem:stableCM}
For a Cohen-Macaulay stable monomial ideal $I$ of codimension $h$,
it follows that
$$
h = \max\set{\max(\m)}{ \m \in G(I)}.
$$
and $x_h^{l_I} \in G(I)$ for a unique positive integer $l_I$.
\end{lem}

Set $G_h(I) := \set{\m \in G(I)}{\max(\m) = h}$ and let $G_h(I) = \bra{\m^{(1)},\dots, \m^{(r)}}$
with $\m^{(1)} \prec \m^{(2)} \ \dots \prec \m^{(r)}$.
Clearly $\m^{(1)} = x_h^{l_I}$.
Let $P_\bullet$ be the Eliahou-Kervaire resolution and $\gG_P$ its associated poset.
For a monomial $\m \in S$, we set
$$
\deg_i(\m) := \max\set{k \ge 0}{x_i^k \text{ divides } \m}.
$$

\begin{lem}\label{lem:basic_lem}
The following hold.

\begin{enumerate}
\item For any $\m \in G(I)$ and $k \in \supp(\m)$, there exists an integer $l$ with $l > 0$ such that
$$
\frac{\m}{x_k} \cdot x_h^l \in G(I).
$$
\label{enum:EKborel_gen}
\item Let $(F,\m)$, $(F',\m')$ be admissible pairs with $F = \bra{ i_1,\dots ,i_q}$,
and let $i_s \in F$. Assume $i_s$ satisfies one of the following conditions:
\begin{enumerate}
\item $\deg_{i_s}(\m') < \deg_{i_s}(\m)$, or
\item $\deg_{i_s}(\m') = \deg_{i_s} (\m)$ and $i_s \not\in F'$.
\end{enumerate}
Then
$$
[\hat 0, (F',\m')] \cap [\hat 0, (F,\m)] \subseteq [\hat 0, (F_{\ang{i_s}},\m)].
$$
In particular, if $\m = \m'$, then $[\hat 0, (F',\m')] \cap [\hat 0, (F,\m)] = [\hat 0, (F \cap F',\m)]$.
\label{enum:EKintersection}
\end{enumerate}
\end{lem}
\begin{proof}
(1) The case $k = h$ is trivial. Assume $k < h$.
Since $x_h^{l_I} \in I$ and $\m/x_k \not\in I$, there exists the least positive integer $l$
such that $(\m/x_k)\cdot x_h^l \in I$. Set $\m' := (\m/x_k) \cdot x_h^l$.
We will show that $\m' \in G(I)$, which completes the proof.
Set $\n := g(\m')$. The equality $\max(\n) = h$ then holds;
otherwise $\n$ divides $\m/x_k$, which is a contradiction.
The equality $\max(\n) = h = \max(\m')$ implies
$\n = (\m/x_k) \cdot x_h^{l'}$ for some positive integer $l'$.
The minimality of $l$ yields the inequality $l' \ge l$, and hence $\n$ is divided by $\m'$.
Therefore it follows that $\m' = \n \in G(I)$, as desired.

(2) We will show only the first assertion; the second is an easy consequence of this assertion. We will make use of the EL-labeling on $\gG_P^\ast$.
Take any $(G,\n) \in [\hat 0, (F',\m')] \cap [\hat 0, (F,\m)] \setminus \bra{\hat 0}$, and let $c, c'$ be the unique unrefinable increasing chain in $[(F,\m),(G,\n)]_{\gG_P^\ast}$ and $[(F',\m'),(G,\n)]_{\gG_P^\ast}$, respectively.
It follows from \eqref{enum:lcm} of Lemma \ref{lem:lcm} that
$$
x_{\gl_+(c)} = \frac{\lcm(\m,\n)}\m, \qquad x_{\gl_+(c')} = \frac{\lcm(\m',\n)}{\m'}.
$$
Suppose $i_s \in \gl_+(c)$. Then $\deg_{i_s}(\n) > \deg_{i_s}(\m)$.
In the case of (a), it follows that $\deg_{i_s}(\n) - \deg_{i_s}(\m') \ge 2$.
However $\deg_{i_s}(\n) - \deg_{i_s}(\m') \le 1$ must hold since $x_{\gl_+(c')}$ is squarefree. This is a contradiction.
In the case of (b), it follows that $i_s \in \gl_+(c')$, which contradicts $i_s \not\in F'$.
Thus $i_s \not\in \gl_+(c)$ in both cases.

If $i_s \not\in G$, then the proof is completed; if this is the case,
then $-i_s \in \gl(c)$ and applying \eqref{enum:unique_C} of Lemma \ref{lem:lcm} yields an unrefinable chain
in $\gG_P^\ast$
starting with $(F,\m) \lessdot^* (F_{\ang{i_s}},\m')$ and ending at $(G,\n)$.
Therefore $(G,\n) < (F_{\ang{i_s}},\m)$ as desired.

In the case (b), the assertion is clear; indeed $G \subseteq F'$.
We will consider the case (a).
Suppose $i_s \in G$. Then $i_s < \max(\n)$ and $i_s \not\in \gl_+(c')$,
and the first implies
\begin{align*}
& \deg_{i_s}(\n) = \deg_{i_s}(g(x_{\gl_+(c)} \m)) = \deg_{i_s}(x_{\gl_+(c)} \m)\\
& \deg_{i_s}(\n) = \deg_{i_s}(g(x_{\gl_+(c')} \m')) = \deg_{i_s}(x_{\gl_+(c')} \m').
\end{align*}
Since $i_s$ does not belong to neither $\gl_+(c)$ nor $\gl_+(c')$,
it follows that
\begin{align*}
\deg_{i_s}(\n)  = \deg_{i_s}(x_{\gl_+(c')} \m')
                   = \deg_{i_s}(\m') 
                  < \deg_{i_s}(\m) 
                  = \deg_{i_s}(x_{\gl_+(c)}\m) = \deg_{i_s}(\n).
\end{align*}
This is a contradiction.
\end{proof}

Set $F^I := \bra{1,\dots ,h-1}$.
It is clear that for each $i$, $(F^I, \m^{(i)})$ is admissible
and $F^I$ is maximal, with respect to inclusion, among the subsets $G \subset \NN$
such that $(G,\m^{(i)})$ is admissible.
Let $\gG_i$ denote the interval $[\hat 0, (F^I,\m^{(i)})]$.

\begin{cor}\label{cor:EK_const}
The following hold.
\begin{enumerate}
\item $\gG = \bigcup_{i=1}^r \gG_i$.
\item $\pr{\bigcup_{i =1}^j \gG_i} \cap \gG_{j+1} = \bigcup_{s \in \supp(\m^{(j+1)}) \setminus \bra{h}}[\hat 0, (F^I_{\ang{s}},\m^{(j+1)})]$.
\label{enum:EKposetint}
\item For any admissible pair $(F, \m)$ with $F := \bra{i_1,\dots ,i_q}$
and for any subset $\gs \subseteq \bra{i_1, \dots ,i_q}$,
the order complex of the poset
$$
\bigcup_{i_r \in \gs} [\hat 0, (F_{\ang{i_r}}, \m)]
$$
is constructible.
\label{enum:EK_const}
\end{enumerate}
\end{cor}
\begin{proof}
(1) It suffices to show the inclusion $\gG \subseteq \bigcup_{i=1}^r \gG_i$.
Obviously $\hat 0$ and $(F,\m)$ with $F \subseteq F^I$ and $\m \in G_h(I)$
belong to $\bigcup_{i=1}^r \gG$.
Take any $(F,\m) \in \gG \setminus \bra{\hat 0}$ with $\m \not\in G_h(I)$.
Clearly $\max(\m) < h$. By \eqref{enum:EKborel_gen} of Lemma \ref{lem:basic_lem},
there exists $\m' \in G_h(I)$ such that $\m = g(x_{\max(\m)} \m')$.
Obviously $(F \cup \bra{\max(\m)}, \m') \in \bigcup_{i=1}^r \gG_i$ and
$(F \cup \bra{\max(\m)}, \m') \gtrdot (F,\m)$. Hence $(F,\m) \in \bigcup_{i=1}^r\gG_i$.

(2) We shall show the inclusion $\supseteq$. Take any $s \in \supp(\m^{(j+1)})$ with $s \neq h$.
By \eqref{enum:EKborel_gen} of Lemma \ref{lem:basic_lem}, there exists a positive integer $l$ such that
$(\m^{(j+1)}/x_s)\cdot x_h^l \in G(I)$. Clearly $(\m^{(j+1)}/x_s) \cdot x_h^l$ is less than
$\m$ with respect to $\succ$, and hence coincides $\m^{(i)}$
for some $i$ with $1 \le i \le j$.
Therefore $(F^I_{\ang{s}},\m^{(j+1)}) = (F^I_{\ang{s}}, g(x_s\m^{(i)})) \in \pr{\bigcup_{i =1}^j \gG_i} \cap \gG_{j+1}$.

To show the inverse inclusion, it suffices to show that each $\gG_i \cap \gG_{j+1}$
is contained in $[\hat 0, (F^I_{\ang s},\m^{(j+1)})]$ for some $s \in \supp(\m^{(j+1)})$ with $s \neq h$.
Since $\m^{(i)} \prec \m^{(j+1)}$,
there exists an integer $s$ with $1 \le s \le n$ such that
$\deg_k(\m^{(i)}) = \deg_k(\m^{(j+1)})$ for $k < s$ and $\deg_s(\m^{(i)}) < \deg_s(\m^{(j+1)})$.
It then follows that $s \in F^I$;
otherwise $\m^{(i)}$ divides $\m^{(j+1)}$, which is a contradiction.
Applying \eqref{enum:EKintersection} of Lemma \ref{lem:basic_lem}, we conclude that
$\gG_i \cap \gG_{j+1} \subseteq [\hat 0, (F^I_{\ang{s}},\m^{(j+1)})]$.

(3) Let $\gD$ be the order complex of $\bigcup_{i_r \in \gs}[ \hat0, (F_{\ang{i_r}},\m)]$.
We will use induction on $\#\gs$.
By \eqref{enum:EKELlabel} of Lemma \ref{lem:ELlabel}, each intervals in $\gG_{P}^\ast$
are shellable and hence so are those in $\gG_P$.
In particular, each intervals in $\gG_P$ are constructible.
Hence the assertion holds when $\#\gs = 1$.
Assume $\#\gs \ge 2$.
Let $\gs = \bra{i_{m_1},\dots ,i_{m_t}}$ and let $\gD', \gD''$ be
the order complexes of
$\bigcup_{s=1}^{t-1} [\hat 0, (F_{\ang{i_{m_s}}},\m)]$ and
$[\hat 0, (F_{\ang{i_{m_t}}},\m)]$.
It suffices to show that $\gD' \cup \gD''$ is constructible.
Both of $\gD'$ and $\gD''$ are constructible by the inductive hypothesis,
and $\dim \gD' = \dim \gD''$.
By Lemma \ref{lem:basic_lem},
$$
\gD' \cap \gD'' =
\gD\pr{\bigcup_{s = 1}^{t-1} [\hat 0, \pr{(F_{\ang{i_{m_t}}})_{\ang{i_{m_s}}}, \m}]}.
$$
Obviously $\dim (\gD' \cap \gD'') = \dim \gD' - 1 = \dim \gD'' - 1$.
Applying the inductive hypothesis shows that $\gD' \cap \gD''$ is also constructible. 
Therefore we conclude that $\gD' \cup \gD''$ is also constructible.
\end{proof}

Now we are prepared to prove the following main theorem
in this section.

\begin{thm}\label{thm:EKball}
Let $I$ be a Cohen-Macaulay stable monomial ideal and $P_\bullet$ its
Eliahou-Kervaire resolution.
Then $P_\bullet$ is supported by a regular CW complex whose underlying space is homeomorphic to a ball.
\end{thm}
\begin{proof}
For simplicity, we set $\gG_0 := \gG_P \setminus \bra{\hat 0}$.
By Corollary \ref{cor:EK_regular}, the poset $\gG_P$ is CW,
and hence coincides with a face poset of a regular CW complex $X$ with the $(-1)$-cell $\void$.
By Proposition \ref{prop:faceposet}, it is enough to show $\gD(\gG_0)$ satisfies the three conditions in Proposition \ref{prop:balllem}.

To verify the constructibility of $\gD(\gG_0)$, by Lemma \ref{lem:const_lem},
we have only to show that $\gD(\gG_P)$ is constructible since
$\gD(\gG_P) = \gD(\gG_0) \ast \ang{\hat 0}$.
The constructibility of $\gD(\gG_P)$ is an immediate consequence of
\eqref{enum:EK_const} of Corollary \ref{cor:EK_const} and the fact that each interval in $\gG_P$ is shellable.

Next we will show that $\gD(\gG_0)$ satisfies the condition \eqref{enum:atmost2}
in Proposition \ref{prop:balllem}.
Note that $\gG_0$ is pure and the maximal length of the chains in $\gG_0$
is equal to $h-1$. Hence $\gD(\gG_0)$ is pure of dimension $h-1$.
Take any $(h-2)$-face $\gs$ of $\gD(\gG_0)$ and
let $c_\gs: c_1 < c_2 < \dots < c_{h-1}$ be the corresponding
chain  in $\gG_0$.
If $c_0$ is not minimal in $\gG_1$ or $c_\gs$ is refinable,
then $c_\gs$ is indeed contained in just two maximal chains in $\gG_1$
since $\gG_P$ is thin.
Assume $c_1$ is minimal and the chain $c_\gs$ is unrefinable
(hence each $<$ in $c_\gs$ is $\lessdot$).
Then every $c_h$ with $c_{h-1} \lessdot c_h$ is a maximal element in $\gG_P$
and of the form $(F^I, \m^{(i)})$.
Let $\m^{(j+1)}$ be the maximal element with respect to $\prec$
such that $c_{h-1} \lessdot (F^I,\m^{(j+1)})$.
If there exists another $(F^I,\m^{(i)})$ with $c_{h-1} \lessdot (F^I,\m^{(i)})$,
then it follows from \eqref{enum:EKposetint} of Corollary \ref{cor:EK_const} that
$g(x_s\m^{(i)})= \m^{(j+1)}$ for some $s \in \supp(\m^{(j+1)} ) \setminus \bra{h}$
and $i \le j$.
This implies $x_s \cdot \m^{(i)} = x_h^k \cdot \m^{(j+1)}$ for suitable integer $k$
with $0 \le k \le l_I$, and hence $\m^{(i)} = g((x_j^{l_I}/x_s)\m^{(j+1)})$.
Moreover if we set $c_{h-1} = (F,\m)$, then $\bra{s} = F^I \setminus F$.
Conseqently $(F^I,\m^{(i)})$ is uniquely determined only by $(F^I,\m^{(j+1)})$ and $c_{h-1}$.
Therefore $c_\gs$ is contained in at most two maximal chains.

Finally, the complex $\gD(\gG_0)$ is indeed satisfies
the condition \eqref{enum:only1} in Proposition \ref{prop:balllem} since the $(h-1)$-face
of the order complex $\gD(\gG_0)$ corresponding to the chain
$$
(F^I_{\ang{h-1}},\m^{(1)}) \gtrdot
((F^I_{\ang{h-1}})_{\ang{h-2}},\m^{(1)}) \gtrdot \dots \gtrdot (\varnothing, \m^{(1)})
$$
is contained only in the facet corrresponding to
$$
(F^I,\m^{(1)}) \gtrdot (F^I_{\ang{h-1}},\m^{(1)}) \gtrdot
((F^I_{\ang{h-1}})_{\ang{h-2}},\m^{(1)}) \gtrdot \dots \gtrdot (\varnothing, \m^{(1)}).
$$
\end{proof}

\begin{exmp}\label{ex:EK}
The CW complex in the figure below supports the Eliahou-Kervaire resolution of
the Cohen-Macaulay Borel fixed ideal $I = (x_1^2,x_1x_2,x_1x_3,x_2^2,x_2x_3,x_3^2)$.
Clearly this is regular and homeomorphic to a $2$-dimensional closed ball.

\begin{figure}[htbp]
\begin{center}
\includegraphics[height=5cm, width=5cm]{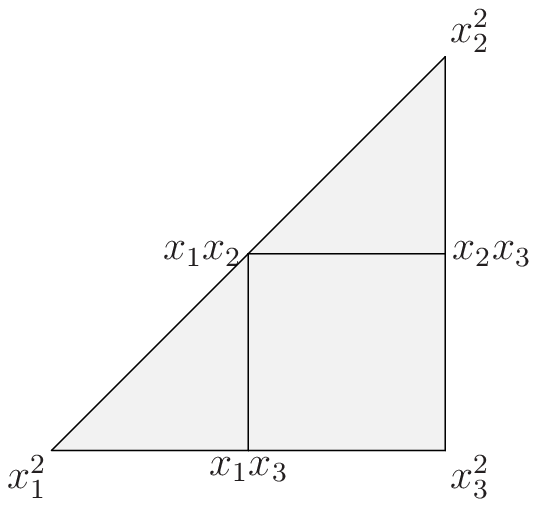}
\end{center}
\caption{}\label{fig:EK}
\end{figure}
\end{exmp}\section{A modified Eliahou-Kervaire resolution of a Borel fixed ideal}\label{sec:regular}

Throughout this section, $I$ denotes a Borel fixed ideal of $S$.
In the paper \cite{OY}, the authors constructed an explicit cellular minimal free resolution
$\wP_\bullet$ of $\bpol(I)$ (see Section~\ref{sec:intro} for the definition of $\bpol(I)$).
Against \cite{OY}, we will call $\wP_\bullet$
the {\em modified Eliahou-Kervaire resolution} and
the term ``Eliahou-Kerviare type resolution" will be used for the generic term of
the usual Eliahou-Kervaire resolution and the modified one.

As is stated in Section~\ref{sec:intro},
the regularity of the CW complex supporting $\wP_\bullet$ is still obscure.
It is noteworthy that we could not deduce the regularity
from Corollary \ref{cor:EK_regular}.
Recall that $\bpol(I)$ is a polarization of $I$ (see Section \ref{sec:intro})
and $Q_\bullet := \wS/(\Theta) \otimes_{\wS} \wP_\bullet$, where
$\Theta := \set{x_{i,1}-x_{i,j}}{1 \le i \le n, \ 2 \le j \le d}$,
becomes a minimal cellular resolution of $I$ with the same supporting CW complex
as $\wP_\bullet$. Of course, $Q_\bullet$ and the Eliahou-Kervaire resolution of $I$
are isomorphic, while their choice of $S$-free basis are different.
This difference gives rise to the one between the supporting CW complexes,
and in general they are not even homeomorphic to each other
(See (3) of Remark~\ref{rem:last_rem} or \cite[Example 6.2]{OY}).

In this section, we will prove that $\wP_\bullet$ is really supported
by a regular CW complex.
First, let us recall the construction of the resolution.

For simplicity, we set $\wI := \bpol(I)$.

\begin{dfn}[{\cite[Definition 2.1]{OY}}]\label{dfn:adm} 
For a finite subset $\wF \subset \NN \times \NN$ 
and a monomial $\m =\prod_{i=1}^e x_{\alpha_i} = \prod_{i=1}^n x_i^{a_i} \in G(I)$ with $1 \le \alpha_1 \le \alpha_2 \le \cdots \le \alpha_e \le n$, 
the pair $(\wF, \wm)$ is said to be {\it admissible} for $\wI$,
if $\wF = \varnothing$, or otherwise $\wF = \bra{\ijseq q}$ with the following
conditions: 
\begingroup
\renewcommand{\labelenumi}{(\alph{enumi})}
\begin{enumerate}
\item $1 \le i_1 < i_2 < \cdots < i_q < \max(\m)$, and
\label{enum:adm1}
\item $j_r = \max \{ \, l \mid \alpha_l \le i_r\, \} +1$ (equivalently, $j_r=1+\sum_{l=1}^{i_r} a_l$) for all $r$. 
\label{enum:adm2}
\end{enumerate}
\endgroup
\end{dfn}

Let $\m \in G(I)$ and $\wF = \bra{\ijseq q}$ with $i_1 < \dots < i_q$.
For integers $s, r$ with $1 \le s < \max(\m)$ and $1 \le r \le q$,
we define the set $\wF_{\ang{i_r}}$, and monomials $\fb_s(\m)$, $\m_{\ang{s}}$, and $\wm_{\ang{s}}$ as follows:
$$
\wF_{\ang{i_r}} := \bra{\ijseq{r-1}, \ijseq[r+1]q},
$$
$\fb_s(\m) := (\m /x_s ) \cdot x_k$, where $k = \min\set{k > s}{k \in \supp(\m)}$,
$\m_{\ang{s}}:= g(\fb_s(\m))$, and $\wm_{\ang{s}}:=\bpol(\m_{\ang{s}})$.
Note that in the above $\fb_s(\m)$ is indeed in $I$ since $I$ is Borel fixed.

Assume $(\wF,\wm)$ is admissible. Then we set
$$
B(\wF,\wm) := \set{i_r}{(\wF_{\ang{i_r}}, \wm_{\ang{i_r}}) \text{ is admissible}}.
$$

\begin{rem}
In \cite{OY}, the set $\wF_{\ang{i_r}}$ is denoted by $\wF_r$.
However we need to know explicitly which element is removed
in the argument in the next section.
Thus we prefer to write $\wF_{\ang{i_r}}$ rather than $\wF_r$.
\end{rem}

To grasp the image of admissible pairs,
it is helpful to draw a diagram of squares as follows:
for any pair $(\wF, \wm)$ with $\wF \subset \NN^2$ and $\wm \in G(\wI)$,
we put a white square in the $(i,j)$-th position for each $(i,j) \in \wF$
and a black square in the $(i',j')$-th position for each $x_{i',j'}$ dividing $\wm$.
If $\wF$ is the maximal subset of $\NN^2$ such that $(\wF, \wm)$
is admissible (such $\wF$ is unique), then the corresponding diagram forms
a ``right-side-down stairs'' with a sole black square in the bottom of each columns.
A pair $(\wF',\wm)$ is then admissible if and only if its diagram is given by
removing some white squares (those corresponding to $\wF \setminus \wF'$)
from the right-side-down stairs of the admissible $(\wF,\wm)$ with $\wF$ maximal.
For example, for $\wm=x_{1,1}x_{1,2}x_{4,3}x_{6,4}x_{6,5}
(= \bpol(x_1^2x_4x_6^2))$, $\wF=\bra{(1,3), (2,3), (3,3), (4,4),(5,4)}$, and
$\wF'=\bra{(1,3), (2,3), (5,4)}$, both of $(\wF,\wm)$ and $(\wF',\wm)$
are admissible and $\wF$ is the maximal subset.
The diagram of $(\wF,\wm)$ becomes as Figure \ref{fig:fm}
and that of $(\wF',\wm)$ as Figure \ref{fig:f'm}

\begin{figure}[h]
\setlength\unitlength{.4mm}
\begin{minipage}{.4\textwidth}
\begin{center}
\begin{picture}(80,80)(0,0)
\thicklines
\put(20,50){\shade\path(0,0)(0,10)(10,10)(10,0)(0,0)}
\put(30,50){\shade\path(0,0)(0,10)(10,10)(10,0)(0,0)}
\put(40,50){\whiten\path(0,0)(0,10)(10,10)(10,0)(0,0)}
\put(40,40){\whiten\path(0,0)(0,10)(10,10)(10,0)(0,0)}
\put(40,30){\whiten\path(0,0)(0,10)(10,10)(10,0)(0,0)}
\put(40,20){\shade\path(0,0)(0,10)(10,10)(10,0)(0,0)}
\put(50,20){\whiten\path(0,0)(0,10)(10,10)(10,0)(0,0)}
\put(50,10){\whiten\path(0,0)(0,10)(10,10)(10,0)(0,0)}
\put(50,0){\shade\path(0,0)(0,10)(10,10)(10,0)(0,0)}
\put(60,0){\shade\path(0,0)(0,10)(10,10)(10,0)(0,0)}
\put(0,0){\makebox(5,60){$i$}}
\put(18,77){\makebox(60,5){$j$}}
\put(10,3){$6$}
\put(10,13){$5$}
\put(10,23){$4$}
\put(10,33){$3$}
\put(10,43){$2$}
\put(10,53){$1$}
\put(23,66){$1$}
\put(33,66){$2$}
\put(43,66){$3$}
\put(53,66){$4$}
\put(63,66){$5$}
\end{picture}
\end{center}
\caption{}\label{fig:fm}
\end{minipage}
\begin{minipage}{.4\textwidth}
\begin{center}
\begin{picture}(80,80)(0,0)
\thicklines
\put(20,50){\shade\path(0,0)(0,10)(10,10)(10,0)(0,0)}
\put(30,50){\shade\path(0,0)(0,10)(10,10)(10,0)(0,0)}
\put(40,50){\whiten\path(0,0)(0,10)(10,10)(10,0)(0,0)}
\put(40,40){\whiten\path(0,0)(0,10)(10,10)(10,0)(0,0)}
\put(40,20){\shade\path(0,0)(0,10)(10,10)(10,0)(0,0)}
\put(50,10){\whiten\path(0,0)(0,10)(10,10)(10,0)(0,0)}
\put(50,0){\shade\path(0,0)(0,10)(10,10)(10,0)(0,0)}
\put(60,0){\shade\path(0,0)(0,10)(10,10)(10,0)(0,0)}
\put(0,0){\makebox(5,60){$i$}}
\put(18,77){\makebox(60,5){$j$}}
\put(10,3){$6$}
\put(10,13){$5$}
\put(10,23){$4$}
\put(10,33){$3$}
\put(10,43){$2$}
\put(10,53){$1$}
\put(23,66){$1$}
\put(33,66){$2$}
\put(43,66){$3$}
\put(53,66){$4$}
\put(63,66){$5$}
\end{picture}
\end{center}
\caption{}\label{fig:f'm}
\end{minipage}
\end{figure}

For an admissible pair $(\wF,\wm)$ with $\wF = \bra{\ijseq q}$
and $(i_r,j_r) \in \wF$, the operation
$(\wF,\wm) \mapsto (\wF_{\ang{i_r}},\wm)$ corresponds to removal of
the $(i_r,j_r)$-th white square from the diagram of $(\wF,\wm)$.
The operation $(\wF,\wm) \mapsto (\wF_{\ang{i_r}},\wm_{\ang{i_r}})$ is a little more complicated.
Recall that $\wm_{\ang{i_r}} = \bpol(g(\fb_{i_r}(\m)))$.
Ignoring the function $g$, the operation
$(\wF,\wm) \mapsto (\wF_{\ang{i_r}},\bpol(\fb_{i_r}(\m)))$
corresponds to removing the $(i_r,j_r)$-th white square
and moving the black square in the bottom of the $j_r$-th column
to the $(i_r,j_r)$-th position.
The diagram $(\wF_{\ang{i_r}},\wm_{\ang{i_r}})$ is given by
removing, from that of $(\wF_{\ang{i_r}},\bpol(\fb_{i_r}(\m)))$,
some black squares successively from the bottom right end.
It is easy to verify that if $(i_r,j_r) \in B(\wF,\wm)$,
then the white $(i_r,j_r)$-th square is in the lowest position
among those white in the $j_r$-th column.

Let $\prec$ be the lexicographic order defined in Section~\ref{sec:EKball}.
The following two lemmas are basic properties of admissible pairs.
See Lemmas 2.2, 2.3, 3.3 and the proof of Proposition 3.1 in \cite{OY}.

\begin{lem}[\cite{OY}]\label{lem:adm} 
Let $(\wF, \wm)$ with $\wF = \bra{\ijseq q}$ be an admissible pair.
For all $k$ with $1 \le k \le q$, the following hold:
\begin{enumerate}
\item $x_{i_k,j_k} = \lcm(\wm, \wm_{\ang{i_k}})/\wm$, and
\item $\m_{\ang{i_k}}$ and $\fb_{i_k}(\m)$ have the same exponents 
in the variables $x_l$ with $l \le i_k$; in particular, $\max(\m_{\ang{i_k}}) \ge i_k$
and $\m_{\ang{i_k}} \succ \m$.
\end{enumerate}
Assume $q \ge 2$. Then for integers $r,s$ with $1 \le r \neq s < q$, the following hold.
\begin{enumerate}
\setcounter{enumi}{2}
\item If $i_r \in B(\wF,\wm)$, then $i_r \in B(\wF_{\ang{i_s}},\wm)$.
\item Assume $i_r, i_s \in B(\wF,\wm)$.
Then $(\wm_{\ang{i_s}})_{\ang{i_r}} = (\wm_{\ang{i_r}})_{\ang{i_s}}$.
Moreover $i_r$ is in $B(\wF_{\ang{i_s}}, \wm_{\ang{i_s}})$ if and only if $i_s$ is in $B(\wF_{\ang{i_r}},\wm_{\ang{i_r}})$.

\item Assume $i_r \not\in B(\wF,\wm)$ and $i_s \in B(\wF,\wm)$.
Then $i_r$ is in $B(\wF_{\ang{i_s}},\wm_{\ang{i_s}})$ if and only if $i_r$ is in
$B(\wF_{\ang{i_s}},\wm)$.
If this is the case, then $\wm_{\ang{i_r}} = (\wm_{\ang{i_s}})_{\ang{i_r}}$.
\end{enumerate}
\end{lem}

Let $A_q^{\wI}$ be the set of all the admissible pairs $(\wF, \wm)$ with $\# \wF = q$,
and $\wP_q$ the free $\wS$-module with basis $\set{e(\wF,\wm)}{(\wF,\wm) \in A_q^{\wI}}$, that is,
$$
\wP_q := \bigoplus_{(\wF, \wm) \in A_q^{\wI}} \wS \cdot e(\wF,\wm),
$$
where the degree of $e(\wF,\wm)$ is set to be that of $\pr{\prod_{(k,l) \in \wF} x_{k,l}} \cdot \wm$.
Define the $\wS$-homomorphism $\partial_q : \wP_q \to \wP_{q-1}$ for $q \ge 1$ so that $e(\wF, \wm)$ 
with $\wF=\bra{\ijseq q}$ is sent to 
$$
\sum_{ 1 \le  r \le q } (-1)^r \cdot  x_{i_r, j_r} \cdot e(\wF_{\ang{i_r}}, \wm) -
\sum_{i_r \in B(\wF, \wm)} (-1)^r \cdot \frac{ x_{i_r, j_r} \cdot \wm}{\wm_{\ang{i_r}}}
\cdot e(\wF_{\ang{i_r}}, \wm_{\ang{i_r}}).
$$
Clearly, each $\partial_q$ is a degree-preserving homomorphism.

\begin{prop}[{\cite[Theorems 2.6 and 5.13]{OY}}]
The $\wS$-free modules $\wP_\bullet$ together with the degree-preserving
homomorphisms $\partial_\bullet$ gives a minimal $\ZZ^{n \times d}$-graded
free cellular resolution of $\bpol(I)$.
\end{prop}

To reach our goal, we will apply a technique similar to Clark's in Section \ref{sec:prel}.
For $(\wF,\wm) \in A_q^{\wI}$ and $(\wF',\wm') \in A_{q-1}^{\wI}$, define
$$
(\wF,\wm) > (\wF', \wm') \ \Iff\  (\wF', \wm') = (\wF_{\ang{i_r}}, \wm)
\text{ or }(\wF_{\ang{i_r}}, \wm_{\ang{i_r}}) \quad \exists i_r \in \wF.
$$
Taking the transitive closure, we obtain the order $<$ on $\bigcup_{q \ge 0} A_q^{\wI}$,
and thus $\bigcup_{q \ge 0} A_q^{\wI}$ becomes a poset.
Let $\gG_{\wP}$ be the poset given by additing the new least element $\hat 0$;
hence $\gG_{\wP} := (\bigcup_{q \ge 0} A_q^{\wI}) \cup \bra{\hat 0}$.

\begin{lem}[{\cite[Proposition 6.1]{OY}}]\label{lem:intv}
The poset $\gG_{\wP}$ is {\em thin}.
\end{lem}

Since the coefficients of the differential maps in $\wP_\bullet$ consists only of $\pm 1$,
we can deduce that $\wP_\bullet$ is supported by a regular CW complex
if the poset $\gG_\wP$ is isomorphic to the face poset of
a regular CW complex.
Thus we shall show $\gG_\wP$ is CW.
To prove this, we will make use of Proposition \ref{prop:CW}.
It is clear that $\#\gG_\wP \ge 2$, and as is stated above,
$\gG_\wP$ has the least element $\hat 0$.
Besides, by Lemma~\ref{lem:intv}, it is thin.
Since $\gG_\wP$ is finite, only the shellability of each interval $[\hat 0, \gs]$
is not trivial.
Let $\gG^*_\wP$ be the {\em dual} poset of $\gG_\wP$.
The order on $\gG^*_\wP$ is denoted by $<^\ast$.
We define the edge labeling $\gl: \cC^1(\gG^*_\wP) \to \ZZ$ as follows:
$\gl((\wF,\wm) \lessdot^\ast \hat 0) = 0$ and
\begin{align*}
\gl((\wF,\wm) \lessdot^\ast (\wF',\wm')) :=\begin{cases}
-i_r & \text{ if $\wF' = \wF_{\ang{i_r}}$ and $\wm' = \wm$} \\
i_r & \text{ if $\wF' = \wF_{\ang{i_r}}$ and $\wm' = \wm_{\ang{i_r}}$,}
\end{cases}
\end{align*}
where $\wF = \bra{\ijseq q}$.
By the definition of $\gl$, any unrefinable chain $c$ can be reconstructed uniquely
from $\gl(c)$ and the first component of $c$,
and if $\gl(c)$ contains $0$, then it always comes in last.

Let $c: c_0 \lessdot^\ast c_1 \lessdot^\ast \cdots \lessdot^\ast c_q$ be
an unrefinable chain with $\gl(c) = (\gl_1,\dots ,\gl_q) \in \ZZ^q$.
The set $\cC^q([c_0,c_q]_{\gG^*_{\wP}})$ then consists of all the unrefinable chains
$c': c'_0 \lessdot^\ast c'_1 \lessdot^\ast \cdots \lessdot^\ast c'_q$ with $c'_0 = c_0$
and $c'_q = c_q$.
For convenience, we will use the following terminology:
\begin{itemize}
\item for an integer $\gl'_r$, we say $\gl_r$ {\em can be replaced by} $\gl_r'$
if
$$
\gl(c') = (\gl_1,\dots ,\gl_{r-1},\gl_r',\gl_{r+1},\dots ,\gl_q),
$$
for some $c' \in \cC^q([c_0,c_q]_{\gG^\ast_{\wP}})$
\item two entries $\gl_s, \gl_r$ with $s < r$ are {\em commutative} if 
$$
\gl(c') = (\gl_1,\dots ,\gl_{s-1}, \gl_r, \gl_{s+1},\dots ,\gl_{r-1} , \gl_s, \gl_{r+1}, \cdots ,\gl_q)
$$
for some $c' \in \cC^q([c_0,c_q]_{\gG^\ast_{\wP}})$, and
\item the entry $\gl_r$ is {\em shiftable} to the $s$-th position if
$$
\gl(c') = (\gl_1,\dots, \gl_{s-1}, \gl_r ,\gl_s, \dots ,\gl_{r-1}, \gl_{r+1}, \dots ,\gl_q)
$$
for some $c' \in \cC^q([c_0,c_q]_{\gG^\ast_{\wP}})$
\end{itemize}

\begin{lem}\label{lem:label_lem}
Let $c \in \cC^q(\gG^*_{\wP})$ and assume $q \ge 2$.
Set $\gl(c) := (\gl_1 ,\dots ,\gl_q)$.
The following hold.
\begin{enumerate}
\item Assume $r \ge 2$ and $\gl_r$ is negative.
Then contiguous entries $\gl_{r-1}, \gl_r$ are commutative.
In particular, any negative entry $\gl_r$ is shiftable to any position in the left.
\label{enum:caseneg}
\item Assume $\gl_{r-1} >0$. If $\gl_{r-1}$ and $\gl_r$ are not commutative
(hence $\gl_r \ge 0$ by \eqref{enum:caseneg}),
then $\gl_{r-1}$ can be replaced by $-\gl_{r-1}$.
\end{enumerate}
\end{lem}
\begin{proof}
Let $c_0 \lessdot^\ast c_1 \lessdot^\ast \cdots \lessdot^\ast c_q$ be the unrefinable chain $c$.
We set $c_{r-2} = (\wF,\wm)$ and $\wF = \bra{\ijseq p}$.
By the definition of the order $<^\ast$,
it follows that $c_{r-1} = (\wF_{\ang{i_k}},\wm')$ and
$c_r = ((\wF_{\ang{i_k}})_{\ang{i_l}}, \wm'')$ for some $k, l$ with $1 \le k \neq l < p$ and some monomials $\wm', \wm'' \in G(\wI)$.

(1) The second assertion is an immediate consequence of the first.
We will show the first assertion.
Note that the negativity of $\gl_r$ implies $c_{r-1} \neq \hat 0$, $c_r \neq \hat 0$,
and $\wm' = \wm''$.
Thus $c_r = ((\wF_{\ang{i_k}})_{\ang{i_l}}, \wm')$.
Obviously $(\wF_{\ang{i_l}} ,\wm)$ is admissible.
Set $c'_{r-1} = (\wF_{\ang{i_l}}, \wm)$.
Then $c_{r-2} \lessdot^\ast c'_{r-1} \lessdot^\ast c_r$ holds, and
it is easy to verify that
$\gl(c_{r-2} \lessdot^\ast c'_{r-1} \lessdot^\ast c_r) = (\gl_r, \gl_{r-1})$.
Thus the unrefinable chain $c'$, given by replacing $c_{r-1}$
by $c'_{r-1}$ in $c$, belongs to $\cC^q([c_0,c_q]_{\gG_{\wP}^\ast})$
and $\gl(c')$ is just the vector given by interchanging $\gl_{r-1}$ and $\gl_r$
in $\gl(c)$.
Therefore $\gl_{r-1}$ and $\gl_r$ are commutative.

(2) The assertion is clear when $\gl_r = 0$.
Assume $\gl_r > 0$. The positivity of $\gl_{r-1}$ and $\gl_r$ imply
$c_r \neq \hat 0$, $c_{r-1} = (\wF_{\ang{i_k}}, \wm_{\ang{i_k}})$,
and $c_r = ((\wF_{\ang{i_k}})_{\ang{i_l}}, (\wm_{\ang{i_k}})_{\ang{i_l}})$.
Hence $i_k \in B(\wF,\wm)$ and $i_l \in B(\wF_{\ang{i_k}}, \wm_{\ang{i_k}})$.
Suppose $i_l \in B(\wF,\wm)$. Then it follows
from Lemma~\ref{lem:adm} that
$i_k \in B(\wF_{\ang{i_l}},\wm_{\ang{i_l}})$ and $(\wm_{\ang{i_k}})_{\ang{i_l}} = (\wm_{\ang{i_l}})_{\ang{i_k}}$.
Thus $c_r = ((\wF_{\ang{i_l}})_{\ang{i_k}}, (\wm_{\ang{i_l}})_{\ang{i_k}})$.
Replacing $c_{r-1}$ by $(\wF_{\ang{i_l}},\wm_{\ang{i_l}})$ in $c$,
we have an unrefinable chain $c'$ whose $\gl(c')$ is equal to the vector
given by interchanging $\gl_{r-1}$ and $\gl_r$ in $\gl(c)$.
This contradicts the hypothesis that $\gl_{r-1}$ and $\gl_r$ are not commutative.
Thus $i_l \not\in B(\wF,\wm)$.
Since $i_l \in B(\wF_{\ang{i_k}},\wm_{\ang{i_k}})$,
applying Lemma~\ref{lem:adm}, we see that
$i_l \in B(\wF_{\ang{i_k}},\wm)$ and $\wm_{\ang{i_l}} = (\wm_{\ang{i_k}})_{\ang{i_l}}$.
Hence $c_{r-2} \lessdot^\ast c'_{r-1} \lessdot^\ast c_r$, where $c'_{r-1} = (\wF_{\ang{i_k}},\wm)$.
Let $c'$ be the unrefinable chain given by replacing $c_{r-1}$ by $c'_{r-1}$ in $c$.
It is easy to show that $\gl(c')$ is equal to the vector
given by replacing $\gl_{r-1} = i_k$ by $-i_k$.
\end{proof}

Let $c: c_0 \lessdot^\ast c_1 \lessdot^\ast \cdots \lessdot^\ast c_q$ be
an unrefinable chain in $\gG^*_{\wP}$.
Assume $c_q \neq \hat 0$.
Set $c_0 = (\wF,\wm)$ and $c_q = (\wF',\wm')$,
and let $\wF \setminus \wF' = \bra{(i_1,j_1), \dots, (i_q,j_q)}$.
We set $\gl_+(c) := \set{k \in \gl(c)}{k > 0}$
and define the squarefree monomial $\wu(c)$ of $\wS$
as follows:
$$
\wu(c) := \prod_{i_r \in \gl_+(c)} x_{i_r,j_r}.
$$
\begin{lem}\label{lem:label_unique}
Let $c: c_0 \lessdot^\ast c_1 \lessdot^\ast \cdots \lessdot^\ast c_q$ be
an increasing unrefinable chain with $c_q \neq \hat 0$,
$c_0 = (\wF,\wm)$, and $c_q = (\wF',\wm')$.
Then it follows that
$$
\wu(c)= \frac{\lcm(\wm,\wm')}{\wm}.
$$
\end{lem}
\begin{proof}
If every entry in $c$ is negative, then $\wu(c) = 1$ and $\wm = \wm'$.
Hence the assertion holds.
Suppose not. Then all the entries in $\gl(c)$ are positive,
or otherwise there exists $r$ with $1 \le r < q$
such that $\gl(c_{i-1} \lessdot^\ast c_i) < 0$ if $1 \le i \le r$
and $\gl(c_{i-1} \lessdot^\ast c_i) > 0$ if $r + 1 \le i \le q$,
since the chain $c$ is increasing.
In the latter case, setting $c_{\ge r}$ to be the increasing chain
$c_r \lessdot^\ast \dots \lessdot^\ast c_q$,
it follows that $\wu(c) = \wu(c_{\ge r})$ and $\wm = \wm''$,
where $\wm''$ is the monomial corresponding to $c_r$.
Thus it suffices to consider only the case where
each entries in $\gl(c)$ are positive.
Assume all the entries in $\gl(c)$ are positive.
Let $\wF \setminus \wF' = \bra{(i_1,j_1),\dots ,(i_q,j_q)}$ with $i_1 < \dots < i_q$.
Since $c$ is increasing, it follows that
$\gl(c) = (i_1, \dots ,i_q)$.
We will show
$$
\wm \cdot \wu(c) = \lcm(\wm,\wm'),
$$
by the induction on $q$, which completes the proof.
When $q = 1$, the assertion follows from Lemma~\ref{lem:adm}.
Assume $q > 1$. Let $c_{\le q-1}$ be the increasing chain
$c_0 \lessdot^\ast \cdots \lessdot^\ast c_{q-1}$.
We set $c_{q-1} = (\wG,\wn)$.
Since $\wu(c) = \wu(c_{\le q-1}) \cdot x_{i_q,j_q}$,
applying the inductive hypothesis, we have the following equalities:
$$
\wm \cdot \wu(c) = \wm \cdot \wu(c_{\le q-1}) \cdot x_{i_q,j_q}
             = \lcm(\wm,\wn) \cdot x_{i_q,j_q}.
$$
Recall that $\wm'=\bpol(\m')$ and $\wn=\bpol(\n)$.
By the definition of an admissible pair, each variable $x_{i_r,j_r}$
with $1 \le r \le q-1$ does not divide $\wm$
and moreover $x_{i_q,j_q}$ does neither $\wm$ nor $\wn$.
Hence $\wu(c_{\le q-1}) = \prod_{s=1}^{q-1} x_{i_s,j_s}$ divides $\wn$.
On the other hand, since $\m' = \n_{\ang{i_q}}$, it follows that
$\deg_k \m' =\deg_k \n$ for all $k < i_q$, and
hence $\wu(c_{\le q-1})$ divides $\wm'$. 

Consequently, we have the following equalities:
\begin{align*}
\lcm(\wm,\wn) \cdot x_{i_q,j_q} &= \lcm(\wm,\wn \cdot x_{i_q,j_q}) \\
     &= \lcm(\wm,\lcm(\wn,\wm')) \\
     &= \lcm(\lcm(\wm,\wn),\wm')\\
     & = \lcm(\wm \cdot \wu(c_{\le q-1}), \wm') = \lcm(\wm, \wm'),
\end{align*}
where the second equality follows from the inductive hypothesis
and the last  from the above remark that $\wu(c_{\le q-1})$ divides $\wm'$ and does not $\wm$.

Summing up, we conclude that $\wm \cdot \wu(c) = \lcm(\wm,\wm').$
\end{proof}

\begin{prop}\label{prop:EL-shellable}
With the above edge-labeling $\gl$,
every interval in $\gG_{\wP}^*$ is EL-shellable,
and hence shellable.
\end{prop}
\begin{proof}
Let $[c_0,c_q]_{\gG_{\wP}^\ast}$ be an interval in $\gG_{\wP}^*$ of length $q$.
When $q \le 1$, then it is obviously shellable.
Assume $q \ge 2$. In the sequel, we tacitly use the following fact:
for any $c$, $c' \in \cC^q([c_0,c_q]_{\gG_{\wP}^\ast})$, the equality $\gl(c) = \gl(c')$
implies $c = c'$.
We divide the arguments into the following two cases.

{\em The case $c_q = \hat 0$}.
Let $c_0 = (\wF,\wm)$ and $\wF = \bra{\ijseq{q-1}}$.
Assume there exists an increasing maximal chain $c$ in $[c_0, \hat 0]_{\gG_{\wP}^\ast}$.
Then the last entry in $\gl(c)$ is $0$.
Since $c$ is increasing, all other entries must be negative,
and hence each of them is equal to $-i_s$ for some $s$.
Therefore $\gl(c)$ must be equal to $(-i_q, -i_{q-1}, \dots ,-i_1, 0)$.
The corresponding chain $c$ is
$$
(\wF,\wm) \lessdot^\ast (\wF_{\ang{i_q}},\wm) \lessdot^\ast
((\wF_{\ang{i_q}})_{\ang{i_{q-1}}}, \wm) \lessdot^\ast \cdots \lessdot^\ast \hat 0,
$$
and this is the unique increasing maximal one in $[c_0,\hat 0]_{\gG_{\wP}^\ast}$.
Obviously $c <_{\lex} c'$ for any other maximal chain in $[c_0, \hat0]_{\gG_{\wP}^\ast}$.

{\em The case $c_q \neq \hat 0$}.
For any maximal chain $c'$, applying Lemma~\ref{lem:label_lem} iteratively
yields an increasing one $c''$ such that $c'' \le_{lex} c'$.
Hence it suffices to show the uniqueness of an increasing maximal chain $c$.
Let $c_0 = (\wF,\wm)$ and $c_q = (\wF',\wm')$.
Set $\wF \setminus \wF' = \bra{\ijseq q}$ and
$\gl_-(c) := \set{k \in \gl(c)}{k < 0}$.
The disjoint union $\gl_+(c) \sqcup \gl_-(c)$ is the set of all the entries in $\gl(c)$,
since $0 \not\in \gl(c)$.
It follows from Lemma~\ref{lem:label_unique} the set $\gl_+(c)$ does not
depend on the choice of the increasing chain $c$ and hence
neither does $\gl_-(c)$.
This implies that the set of all the entries in $\gl(c)$
is independent of the choice of $c$,
and hence so is $\gl(c)$ itself, since $c$ is increasing.
Thus we conclude that $c$ is unique.
\end{proof}

Applying Proposition \ref{prop:CW}, we obtain the following.

\begin{cor}\label{cor:EKtype_reg}
Our poset $\gG_\wP$ is CW. In particular, the resolution $\wP_\bullet$ is supported by
a regular CW complex.
\end{cor}

\section{A regular CW complex supporting a modified Eliahou-Kervaire resolution  of a Cohen-Macaulay Borel fixed ideal}

Continuously, let $I$ be a Borel fixed ideal of $S$.
By Corollary \ref{cor:EKtype_reg},
the resolution $\wP_\bullet$ of $I$ is supported by a regular CW complex as the Eliahou-Kervaire resolution is, which leads us to expect that an analogue of Theorem \ref{thm:EKball}
holds true for $\wP_\bullet$.
In this section, we will show the following:

\begin{thm}\label{thm:ball}
Assume $I$ is Cohen-Macaulay. Then the resolution $\wP_\bullet$ is
supported by a regular CW complex whose underlying space is
homeomorphic to a closed ball.
\end{thm}

Although the idea of the proof of this assertion is quite the same as that of Theorem \ref{thm:EKball},
there are subtle differences between the proofs of the corresponding
preliminary lemmas for Theorems \ref{thm:EKball} and \ref{thm:ball}.
That's why we will give their proofs.

Henceforth we assume $S/I$ is Cohen-Macaulay of codimension $h$.
It follows from Lemma \ref{lem:stableCM} that
$$
h = \max\set{\max(\m)}{ \m \in G(I)}.
$$
and $x_h^{l_I} \in G(I)$ for a unique positive integer $l_I$.
As in Section \ref{sec:EKball},
let $\prec$ to be the lexicographical order on the monomials in $S$
with respect to $x_1 \succ x_2 \succ \cdots \succ x_n$, and
let $G_h := \bra{\m^{(1)},\dots ,\m^{(r)}}$ with
$\m^{(1)} \prec \cdots \prec \m^{(r)}$ be the set of all the elements $\m \in G(I)$ such that $\max(\m) =h$.
For each $i$, let $(\wF^{(i)},\wm^{(i)})$, where $\wm^{(i)} = \bpol(\m^{(i)})$,
be the {\em full} admissible pair,
i.e., $\# \wF^{(i)} = h - 1$.
Obviously, $\m^{(1)} = x_h^{l_I}$ and $\wF^{(1)} = \bra{(1,1),\dots ,(h-1,1)}$.
We set
$$
\gG_i := [\hat 0, (\wF^{(i)}, \wm^{(i)})].
$$

\begin{lem}\label{lem:const}
The following hold.
\begin{enumerate}
\item For any $\m \in G(I)$ and $k \in \supp(\m)$ with $k < h$, there exists an integer $l$ with
$l \ge 0$ such that
$$
\frac{\m}{x_k} \cdot x_{k+1} \cdot x_h^l \in G(I)\text{;}
$$
in particular, for any $\m \in G(I) \setminus G_h$, we can choose a positive integer
$l$ to satisfy $(\m/x_{\max(\m)}) \cdot x_h^l \in G(I)$
and hence $(\m/x_{\max(\m)}) \cdot x_h^l \in G_h$.
\item Let $(\wF, \wm), (\wF', \wm')$ be admissible paris 
with $\wF = \bra{\ijseq q}$. If there exists $(i_s, j_s) \in \wF \setminus \wF'$ such that
$x_{i_s,j_s} \nmid \wm'$, then
$$
[\hat 0, (\wF,\wm)] \cap [\hat 0, (\wF', \wm')] \subseteq [\hat 0, (\wF_{\ang{i_s}}, \wm)].
$$
In particular, if $\wm' = \wm$, then
$$
[\hat 0, (\wF,\wm)] \cap [\hat 0, (\wF', \wm')] = [\hat 0, (\wF \cap \wF', \wm)].
$$
\end{enumerate}
\end{lem}
\begin{proof}
(1) The second assertion is an easy consequence of the first.
We will show the first assertion.
Set $\n := (\m/x_k) \cdot x_{k+1}$. Note that $\n \cdot x_h^{l_I} \in I$.
Let $l$ be the least nonnegative integer $l$ such that $\n \cdot x_h^l \in I$.
We have only to show $\n \cdot x_h^l \in G(I)$ to complete the proof.
If $l = 0$, then $\n$ indeed must belong to $G(I)$;
otherwise there exists $\n' \in G(I)$ which strictly divides $\n$,
and applying $\fb_k(-)$ to $\n'$,
we obtain a monomial $\n'' \in I$ dividing $\m$ strictly.
This is a contradiction.

Now assume $l \ge 1$ and suppose $\n \cdot x_h^l \not\in G(I)$.
Then there exists a monomial
$\m' \in G(I)$ which strictly divides $\n \cdot x_h^l$.
However applying suitable operators $\fb_s(-)$ to $\m'$,
we can construct the monomial $\n'' \cdot x_h^{l'} \in I$ with
$\n'' \mid \m$ and $0 \le l' < l$, and hence it follows that $\m \cdot x_h^{l'} \in I$.
This contradicts the minimality of $l$.

(2) The second assertion is an immediate consequence of the first.
To prove the first assertion, we will make use of the edge labeling
$\gl: \cC^1(\gG_{\wP}^*) \to \ZZ$ defined in the previous section.
It is clear that $\hat 0 \in [\hat 0, (\wF_{\ang{i_s}},\wm)]$.
Take any $(\wF'', \wm'') \in [\hat 0, (\wF',\wm')] \cap [\hat 0, (\wF, \wm)] \setminus \bra{\hat 0}$.
It is enough to show that $(\wF_{\ang{i_s}}, \wm) <^* (\wF'',\wm'')$.
Let $c,c'$ be the unique maximal increasing chains in the intervals
$[(\wF, \wm), (\wF'', \wm'')]_{\gG_{\wP}^\ast}$ and $[(\wF',\wm'), (\wF'', \wm'')]_{\gG_{\wP}^\ast}$ in
$\gG_{\wP}^\ast$, respectively.
By Lemma~\ref{lem:label_unique},
$$
\wu(c) = \frac{\lcm(\wm, \wm'')}{\wm}, \quad \wu(c') = \frac{\lcm(\wm',\wm'')}{\wm'}.
$$
Suppose $x_{i_s, j_s} \mid \wu(c)$.
From the hypothesis $(i_s, j_s) \in \wF$ and the definition of admissible pair,
it follows that $x_{i_s,j_s} \nmid \wm$, and hence $x_{i_s,j_s} \mid \wm''$.
This in turn implies $x_{i_s, j_s} \mid \wu(c')$ since $x_{i_s, j_s} \nmid \wm'$.
However it then follows that $(i_s, j_s) \in \wF'$. This is a contradiction.
Therefore $x_{i_s,j_s}$ does not divide $\wu(c)$.
On the other hand, $(i_s,j_s) \in \wF \setminus \wF'$ also implies
$(i_s,j_s) \in \wF \setminus \wF''$. Hence it follows that $-i_s \in \gl(c)$.
Since every negative entry is shiftable to any position in the left
by Lemma~\ref{lem:label_lem},
we have the following chain:
$$
(\wF,\wm) \lessdot^\ast (\wF_{\ang{i_s}},\wm) \lessdot^\ast \cdots \lessdot^\ast (\wF'',\wm'').
$$
Therefore $(\wF_{\ang{i_s}}, \wm) <^\ast (\wF'', \wm'')$.
\end{proof}

\begin{cor}\label{cor:const}
The following hold.

\begin{enumerate}
\item $\gG_{\wP} = \bigcup_{i=1}^r \gG_i$; hence the poset $\gG_{\wP}$ is {\em pure}.
\item The maximal elements of $\pr{\bigcup_{i=1}^j \gG_i} \cap \gG_{j+1}$
are
$$
\set{(\wF^{(j+1)}_{\ang{s}},\wm^{(j+1)})}{s \in \supp(\m^{(j+1)}) \setminus \bra{h}}.
\label{enum:max_elems}
$$
\item For any admissible pair $(\wF, \wm)$ with $\wF := \bra{\ijseq q}$
and for any subset $\gs \subseteq \bra{i_1, \dots ,i_q}$,
the order complex of the poset
$$
\bigcup_{i_r \in \gs} [\hat 0, (\wF_{\ang{i_r}}, \wm)]
$$
is constructible. \label{enum:union_const}
\end{enumerate}
\end{cor}
\begin{proof}
The assertion \eqref{enum:union_const} can be shown by the same argument as
\eqref{enum:EK_const} of Corollary \ref{cor:EK_const}.
We will show the assertions (1) and (2).

(1) Only the inclusion $\gG_{\wP} \subseteq \bigcup_{i=1}^r \gG_i$
is not trivial, and it suffices to show that $\gG_{\wP} \setminus \hat 0$
is contained in $\bigcup_{i=1}^r \gG_i$.
Let $(\wF,\wm) \in \gG_\wP \setminus \hat 0$. If $\max(\m) = h$, then $\wm \in G_h$,
and hence $(\wF, \wm) \in \bigcup_{i=1}^j \gG_i$.
Assume $\max(\m) < h$. Then by Lemma~\ref{lem:const}, there exists $\m^{(i)} \in G_h$ such that
$\m^{(i)}_{\ang{i_k}} = \m$, where $i_k = \max(\m)$.
By Lemma~\ref{lem:adm}, the pair $(\wF, \wm^{(i)})$ is admissible.
Let $j_k$ be the integer such that $(\bra{(i_k,j_k)},\wm^{(i)})$ is admissible.
Set $\wG = \wF \cup \bra{(i_k,j_k)}$.
Obviously, $(i_k,j_k) \not\in \wF$ and $(\wG, \wm^{(i)})$ is admissible.
Since $(\wG, \wm^{(i)}) \in \bigcup_{i=1}^r \gG_i$
and $(\wF,\wm) = (\wG_{\ang{i_k}},\wm^{(i)}_{\ang{i_k}})$,
it follows that $(\wF, \wm) \in \bigcup_{i=1}^r \gG_i$.

(2) First, we will show that each $(\wF^{(j+1)}_{\ang{s}},\wm^{(j+1)})$ with $s \in \supp(\m^{(j+1)}) \setminus \bra{h}$
is indeed in $\bigcup_{i=1}^j \gG_i \cap \gG_{j+1}$.
We have only to verify that $(\wF^{(j+1)}_{\ang{s}},\wm^{(j+1)}) \in \gG_i$
for some $i$.
According to Lemma~\ref{lem:const}, there exists a nonnegative integer $l$ such that
$(\m^{(j+1)}/x_s) \cdot x_{s+1} \cdot x_h^l \in G(I)$.
This generator is in $G_h$, since $\m^{(j+1)} \in G_h$ and $s \neq h$,
and thus it equals $\m^{(i)}$ for some $i$ with $i \le j$.
Moreover $\m^{(i)}_{\ang{s}} = \m^{(j+1)}$ and $\wF^{(i)}_{\ang{s}} = \wF^{(j+1)}_{\ang{s}}$.
Therefore $(\wF^{(j+1)}_{\ang{s}}, \wm^{(j+1)}) \in \gG_i$.

To show the inverse inclusion, it is enough to verify that
$$
\gG_i \cap \gG_{j+1} \subseteq [\hat 0, (\wF^{(j+1)}_{\ang{s}}, \wm^{(j+1)})]
$$
for any $i$ with $1 \le i \le j$ and for some $s \in \supp(\wm^{(j+1)}) \setminus \bra{h}$.
Let $\wF^{(j+1)} = \bra{(1,j_1), \dots ,(h-1,j_{h-1})}$.
Since $\wm^{(i)} \prec \wm^{(j+1)}$, there exists an integer $s_0$ such that
$\deg_k(\m^{(i)}) = \deg_k(\m^{(j+1)})$ for $k < s_0$ and
$\deg_{s_0}(\m^{(i)}) < \deg_{s_0}(\m^{(j+1)})$.
Obviously it follows that $s_0 \in \supp\pr{\m^{(j+1)}} \setminus \bra{h}$,
$(s_0, j_{s_0}) \not\in \wF^{(i)}$, and that $x_{s_0,j_{s_0}} \nmid \wm^{(i)}$.
Applying Lemma~\ref{lem:const}, we conclude that
$\gG_i \cap \gG_{j+1} \subseteq [\hat 0, (\wF^{(j+1)}_{\ang{s_0}}, \wm^{(j+1)})]$.
\end{proof}

\begingroup
\renewcommand\proofname{Proof of Theorem~\ref{thm:ball}}
\begin{proof}
As in the proof of Theorem~\ref{thm:EKball},
it suffices to show that $\gD(\gG_{\wP} \setminus \bra{\hat 0})$
satisfies the conditions \eqref{enum:const}, \eqref{enum:atmost2}, and \eqref{enum:only1}
in Proposition \ref{prop:balllem}.

The proof that $\gD(\gG_{\wP} \setminus \bra{\hat 0})$ satisfies \eqref{enum:const}
is the same as that of the corresponding assertion of Theorem \ref{thm:EKball}
 (use Lemma \ref{lem:const_lem} and Corollary \ref{cor:const}).
Thus we shall verify that $\gD(\gG_{\wP} \setminus \bra{\hat 0})$ satisfies
the conditions \eqref{enum:atmost2} and \eqref{enum:only1}.

As for \eqref{enum:atmost2},
since $\gG_{\wP}$ is thin, by the same argument as in the proof of
Theorem \ref{thm:EKball}, it suffices to show that
each unrefinable chain $c: c_1 \lessdot \dots \lessdot c_{h-1}$ with $c_1$ minimal
in $\gG_{\wP} \setminus \bra{\hat 0}$ is contained in at most 2 maximal chains.

Note that every element $c_h$ with $c_{h-1} \lessdot c_h$
is the maximal one in $\gG_{\wP}$ and of the form $(F^{(i)},\wm^{(i)})$.
Let $\wm^{(j+1)}$ be a maximal element with respect to $\prec$
such that $c_{h-1} \lessdot (\wF^{(j+1)},\wm^{(j+1)})$.
If there exists another maximal element $(\wF^{(i)}, \wm^{(i)})$ in $\gG_{\wP}$
such that $c_{h-1} \lessdot (\wF^{(i)},\wm^{(i)})$, by Corollary \ref{cor:const}, there exists
$s \in \supp(\m^{(j+1)}) \setminus \bra{h}$ such that
$\wF^{(i)}_{\ang{s}} = \wF^{(j+1)}_{\ang{s}}$ and $\wm^{(i)}_{\ang{s}} = \wm^{(j+1)}$.
By the choice of $(\wF^{(j+1)},\wm^{(j+1)})$, it follows that $i \le j$.
These implies that $(\m^{(i)} \cdot (x_s/x_{s+1})) = \m^{(j+1)} \cdot x_h^k$
for suitable non-negative integer $k$,
and hence $\wm^{(i)} = \bpol(g(x_h^{l_I} \cdot (x_{s+1}/x_s) \cdot \m^{(j+1)} ))$.
Moreover if we set $c_{h-1} = (\wF, \wm)$, then
$$
\bra{s} = \bra{1,\dots, h-1} \setminus \set{k}{(k,l) \in \wF \ \text{for some $l$}}.
$$
Thus $(\wF^{(i)},\wm^{(i)})$ is uniquely determined by $(\wF^{(j+1)},\wm^{(j+1)})$
and $c_{h-1}$.
Hence $c$ is contained in at most $2$ maximal chains.

What remains to be proved is that $\gD(\gG_{\wP} \setminus \bra{\hat 0})$
satisfies \eqref{enum:only1}.
This is clear: indeed, the face corresponding to the chain
$$
(\wF^{(1)}_{\ang{h-1}}, \wm^{(1)}) \gtrdot ((\wF^{(1)}_{\ang{h-1}})_{\ang{h-2}}, \wm^{(1)}) \gtrdot \cdots \gtrdot (\varnothing ,\wm^{(1)})
$$
is contained only the facet corresponding to
the one
$$
(\wF^{(1)},\wm^{(1)}) \gtrdot (\wF^{(1)}_{\ang{h-1}}, \wm^{(1)}) \gtrdot ((\wF^{(1)}_{\ang{h-1}})_{\ang{h-2}}, \wm^{(1)}) \gtrdot \cdots \gtrdot (\varnothing ,\wm^{(1)}).
$$
\end{proof}
\endgroup

\begin{exmp}\label{ex:modEK}
Let $I$ be the same monomial ideal as in Example~\ref{ex:EK}.
The following CW complex supports the modified Eliahou-Kervaire resolution $\wP_\bullet$
of $\wI = \bpol(I)$.
It is indeed regular and homeomorphic to a $2$-dimensional closed ball.

\begin{figure}[htbp]
\begin{center}
\includegraphics[height=5cm, width=5cm]{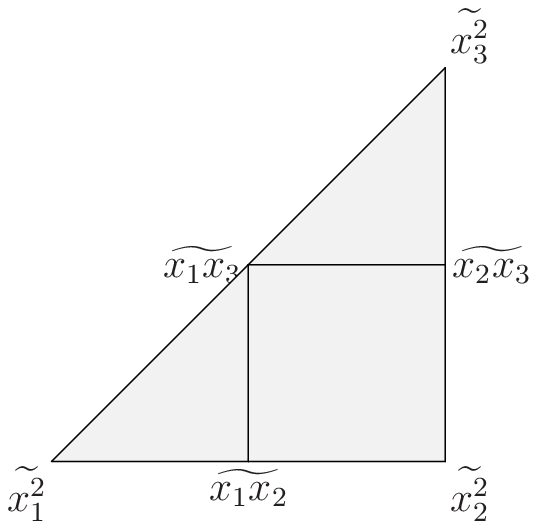}
\end{center}
\caption{}\label{fig:modEK}
\end{figure}
\end{exmp}\section{Final remarks}

Let $I$ be a Borel fixed ideal $I$ (not necessarily Cohen-Macaulay),
and $d$ a positive integer with $d \ge \max\set{\deg(\m)}{\m \in G(I)}$.
Set $N := n + d - 1$ and $T :=  \kk[x_1,\dots ,x_N]$.
For a monomial $\m = \prod_{i=1}^{\deg(\m)} x_{\ga_i} \in G(I)$ with $1 \le \ga_1 \le \dots \le \ga_{\deg(\m)}$,
define the monomial $\m^\gs$ in $T$
as follows:
$$
\m^\gs := x_{\ga_1}x_{\ga_2 + 1}\cdots x_{\ga_i + (i-1)} \cdots x_{\ga_d + \deg(\m) -1}.
$$
Let $I^\gs$ be the monomial ideal of $T$ generated by $\set{\m^\gs}{\m \in G(I)}$.
It is well-known that the monomial ideal $I^\gs$ is then squarefree strongly stable
and any squarefree strongly stable monomial ideal is of the form $I^\gs$ for some
Borel fixed ideal $I$.
As is stated in Section~\ref{sec:intro}, the subset
$\Theta' := \set{x_{i,j} - x_{i+1,j-1}}{1 \le i < n,\ 1 < j \le d}$ forms a regular sequence on $\wS$ and $(\Theta')$ is the kernel of the surjective ring homomorphism
$\wS \to T$ sending $x_{i,j}$ to $x_{i+j-1}$; in particular $\wS/(\Theta') \cong T$.
Moreover through this ring isomorphism,
it follows that $\wS/(\Theta') \otimes_{\wS} \wI \cong I^\gs$ and
$\wS/(\Theta') \otimes_{\wS} \wP_\bullet$ is a minimal $\ZZ^N$-graded
free resolution of $I^\gs$ \cite{Y}.
Obviously $\wS/(\Theta') \otimes_{\wS} \wP_\bullet$ is also cellular and supported by
the same CW complex as $\wP_\bullet$.
Therefore the following is an immediate consequence of Corollary~\ref{cor:EKtype_reg}
and Theorem~\ref{thm:ball}

\begin{cor}\label{cor:Isigma}
With the above notation, the minimal free resolution $\wS/(\Theta') \otimes_{\wS} \wP_\bullet$ of $I^\gs$ is supported by a regular CW complex,
and when $I$ is Cohen-Macaulay, the regular CW complex can be chosen
so that its underlying space is homeomorphic to a closed ball.
\end{cor}

\begin{rem}\label{rem:last_rem}
(1) If a Borel-fixed ideal $I$ is not Cohen-Macaulay, the regular CW complex
supporting $\wP_\bullet$ may not be homeomorphic to a closed ball.
Indeed, let $I = (x_1^2, x_1x_2, x_1x_3, x_2^3, x_2^2x_3)$.
An easy computation then shows that the CW complex can be chosen to be
the following simplicial complex with two triangles glued along a vertex (Figure \ref{fig:tri-tri}). Clearly it is not homeomorphic to a closed ball.

(2) On the other hand, there is a Borel-fixed ideal $I$
which is not Cohen-Macaulay but the regular CW complex supporting $\wP_\bullet$ is a closed ball.
For example, let $I = (x_1^2, x_1x_2, x_1x_3, x_2^2, x_2x_3)$.
The CW complex is in turn given by gluing
a square with a triangle along an edge (Figure \ref{fig:tri-sq}),
and is homeomorphic to a closed ball.
\begin{figure}[htbp]
\begin{minipage}{.48\textwidth}
\begin{center}
\includegraphics[height=4.4cm, width=5cm]{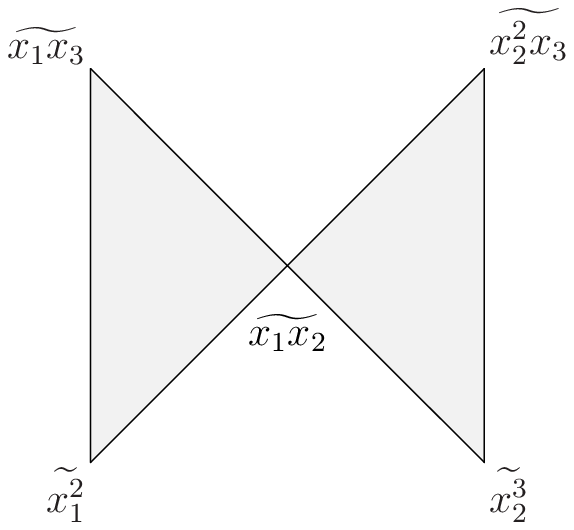}
\end{center}
\caption{}\label{fig:tri-tri}
\end{minipage}
\begin{minipage}{.48\textwidth}
\begin{center}
\includegraphics[height=4.4cm, width=6cm]{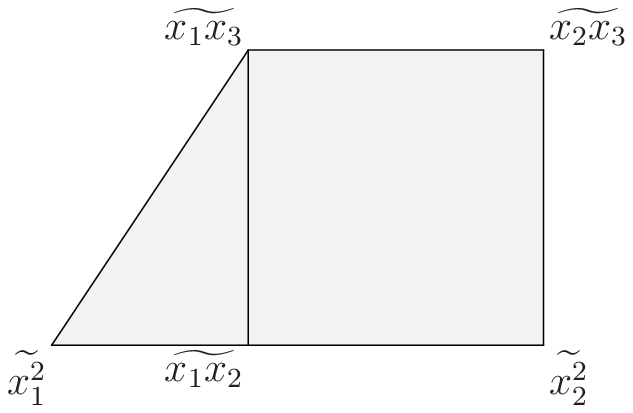}
\end{center}
\caption{}\label{fig:tri-sq}
\end{minipage}
\end{figure}

(3) Recall that, for a Borel fixed ideal $I$,
$\wI$ is a polarization of $I$, and $Q_\bullet := \wS/(\Theta) \otimes_{\wS} \wP_\bullet$
is a minimal free resolution of $I$ with the same supporting CW complex as $\wP_\bullet$
(see Section~\ref{sec:intro} for the definition of $\Theta$).
In the case of (2), the regular CW complex supporting
the Eliahou-Kervaire resolution of $I$ is different from that supporting
$Q_\bullet$.
Indeed, the complex supporting the Eliahou-Kervaire resolution is the simplicial complex
as in Figure~\ref{fig:tri-tri2}, which is not homeomorphic to
the CW complex in Figure~\ref{fig:tri-sq}.
\begin{figure}[htbp]
\begin{center}
\includegraphics[height=4.4cm, width=5cm]{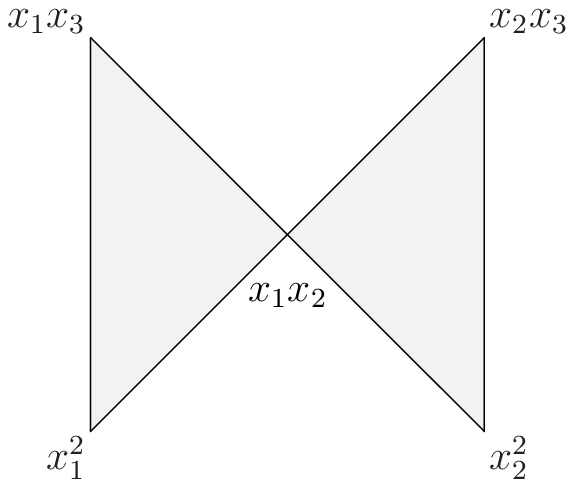}
\end{center}
\caption{}\label{fig:tri-tri2}
\end{figure}

As for a Cohen-Macaulay case, let $I = (x_1^2,x_1x_2,x_1x_3,x_2^2,x_2x_3,x_3^2)$.
The CW complexes supporting its Eliahou-Kervaire resolution and $Q_\bullet$
are described in Examples \ref{ex:EK} and \ref{ex:modEK}, respectively.
These are the same as complexes, while they differ in the labelings
even if the operation `` $\widetilde{\ }$ '' is ignored.

In the case where $I = (x_1^2,x_1x_2,x_1x_3,x_2^4,x_2^3x_3,x_2^2x_3^2,x_2x_3^3,x_3^4)$,
the corresponding CW complexes are different even as complexes.
We leave it the reader to verify this fact.
\end{rem}

As far as the authors have calculated, each regular CW complex
supporting the modified Eliahou-Kervaire resolution of a Borel fixed ideal
can be chosen to be {\em polytopal}, i.e., a regular CW complex
whose cells are polytopes.
This leads us to pose the following question.

\begin{ques}
Is every modified Eliahou-Kervaire resolution $\wP_\bullet$ of $\bpol(I)$ of a Borel fixed ideal $I$
supported by a polytopal complex?
\end{ques}

According to \cite{NR} (see also \cite{OY}), this assertion holds true
if $I$ is generated by monomials of the same degree.

\section*{Acknowledgement}

We are grateful to Professor Satoshi Murai for his valuable comments.

\begin{thebibliography}{88}
\bibitem{AHH2} A. Aramova, J. Herzog and T. Hibi, Shifting operations and graded Betti numbers, 
J. Alg. Combin. {\bf 12} (2000) 207--222. 
%
%
\bibitem{BaWe} E. Batzies and V. Welker,
{\em Discrete Morse theory for cellular resolutions},
J. Reine Angew. Math. {\bf 543} (2002), 147--168.
%
\bibitem{BH} W. Bruns and J. Herzog, {\em Cohen-Macaulay rings},
revised edition, Cambridge University Press, 1998.
%
\bibitem{BPS} D. Bayer, I. Peeva, and B. Sturmfels,
{\em Monomial resolutions},
Math. Res. Lett. {\bf 5} (1998), 31--46.
%
\bibitem{BS} D. Bayer and B. Sturmfels,
{\em Cellular resolutions of monomial modules},
J. reine angew. Math. {\bf 502} (1998), 123--140.
%
\bibitem{B80} A. Bj\"orner,
{\em Shellable and Cohen-Macaulay partially ordered sets},
Trans. Amer. Math. Soc. {\bf 260} (1980), 159--183.
%
\bibitem{B84} A. Bj\"orner,
{\em Posets, regular CW complexes and Bruhat order},
European J. Combin. {\bf 5} (1984), 7--16.
%
\bibitem{B95} A. Bj\"orner,
{\em Topological Methods}, in ``Handbook of Combinatorics'', vol. 2,
1819--1872.
%
\bibitem{BjWa} A. Bj\"orner and M. Wachs,
{\em On lexicographically shellable posets},
Trans. Amer. math. Soc. {\bf 277} (1983), 323--341.
%
\bibitem{C} T. Clark,
{\em A minimal poset resolution of stable ideals}.
In: Progress in Commutative Algebra I: Combinatorics and
Homology, pp. 143--166, de Gruyter, Berlin, 2012.
%
\bibitem{DK} G. Danaraj and V. Klee,
{\em Shellings of spheres and polytopes},
Duke Math. J. {\bf 41}, 443--451.
%
\bibitem{EK} S. Eliahou and M. Kervaire,
{\em Minimal resolutions of some monomial ideals},
J. Algebra {\bf 129} (1990), 1--25. 
%
\bibitem{F} R. Forman, Morse theory for cell complexes, Adv. in Math., 134 (1998), 90--145.
%
\bibitem{HT} J. Herzog and Y. Takayama,
{\em Resolutions by mapping cones}, Homol. Homotopy Appl. {\bf 4} (2002), 277--294.
%
\bibitem{LW} A. T. Lundell and S. Weingram,
{\em The topology of CW complexes}, Van Nostrand Reinhold Company, 1969.
%
\bibitem{M} J. Mermin,
{\em The Eliahou-Kervaire resolution is cellular},
J. Commut. Algebra {\bf 2} (2010), 55--78.
%
\bibitem{NR} U. Nagel and  V. Reiner,
{\em Betti numbers of monomial ideals and shifted skew shapes},
Electron. J. Combin. {\bf 16} (2), 59 (2009).
Special volume in honor of Anders Bj\"orner, Reserach Paper 3.
%
\bibitem{OY} R. Okazaki and K. Yanagawa,
{\em Alternative polarization of Borel fixed ideals, Eliahou-Kervaire type resolution,
and discrete Morse theory}, to appear in J. Alg. Combin.
%
\bibitem{RW} V. Reiner and V. Welker,
{\em Linear syzygies of Stanley-Reisner ideals},
Math. Scand. {\bf 89} (2001), 117--132.
%
\bibitem{S} A. Sinefakopoulos,
{\em On Borel-fixed ideals generated in one degree},
J. Algebra {\bf 319} (2008),2739--2760.
%
\bibitem{St} R. P. Stanley, {\em Combinatorics and commutative algebra},
Birkh{\"a}user, Boston, 2004.
\bibitem{V} M. Velasco,
{\em Minimal free resolutions that are not supported by a CW-complex},
J. Algebra {\bf 319} (2008), 102--114.
%
\bibitem{Y} K. Yanagawa,
{\em Alternative polarizations of Borel fixed ideals},
Nagoya. Math. J. {\bf 207} (2012), 79--93.
%
\end{thebibliography}
\end{document}